\documentclass[11pt]{article}

\usepackage{amsmath,amssymb,amsthm}
\usepackage{graphicx}
\usepackage{booktabs}
\usepackage{geometry}
\usepackage[title]{appendix}%
\usepackage{textcomp}%
\usepackage{manyfoot}%
\usepackage{algorithmicx}%
\usepackage{listings}%
\usepackage{lmodern}
\usepackage{enumitem}
\usepackage{amsfonts}

\usepackage{xcolor}
\usepackage{color}
\usepackage{floatrow}
\usepackage{makeidx}
\usepackage{float}
\usepackage{mathtools}
\usepackage{commath}
\usepackage{graphicx}

\usepackage{hyperref}  
\usepackage{url}       
\usepackage{booktabs}  
\usepackage{microtype} 
\usepackage{multirow}
\usepackage{makecell}
\usepackage{algorithm}
\usepackage{algpseudocode}
\usepackage{algcompatible}
\usepackage{siunitx}
\usepackage{subcaption}
\usepackage{soul}
\usepackage{comment}

\usepackage{nicefrac}  
\usepackage{xfrac}
\usepackage[numbers]{natbib}

\newtheorem{theorem}{Theorem}[section]
\newtheorem{lemma}[theorem]{Lemma}
\newtheorem{proposition}[theorem]{Proposition}
\newtheorem{corollary}[theorem]{Corollary}
\theoremstyle{remark}
\newtheorem{remark}[theorem]{Remark}
\newtheorem{definition}[theorem]{Definition}

\newcommand{\spann}[1]{\operatorname{span}\lbrace #1 \rbrace}

\title{Anchor-Based Function Extrapolation with Proven Bounds and Projection Guarantees}

\author{
Guy Hay\thanks{School of Mathematical Sciences, Tel Aviv University. \texttt{guyhay@mail.tau.ac.il}}
\and
Nir Sharon\thanks{School of Mathematical Sciences, Tel Aviv University. \texttt{nir.sharon@math.tau.ac.il}}
}

\date{} 

\begin{document}
\maketitle

\begin{abstract}
Classical approximation and learning methods are typically optimized for interpolation over a sampled domain $\Omega$, with no guarantees on their behavior in an extrapolation region $\Xi$, where small in-domain errors may amplify. We develop a model-agnostic framework that recasts extrapolation as a feasibility and projection problem with rigorous guarantees. The approach is built around anchor functions, auxiliary constructions for which one can certify an upper bound on the $\Xi$-distance to the unknown target function. Such certificates define feasible sets that are proven to contain the true function. Given any baseline approximation (e.g., least-squares or regularized regression), we obtain a corrected extrapolation by projecting the baseline onto the feasible set; the resulting predictor is proven not to increase the error on $\Xi$, and we prove quantitative bounds on the improvement. We establish new stability constants governing extrapolation, including a tight spectral condition number and a numerically stable inner-domain bound that connects in-domain error to extrapolation risk. To reduce conservatism of worst-case certification, we also propose probabilistic anchor functions that yield high-confidence feasible sets. Numerical experiments, including geomagnetic field modeling and nonlinear oscillators, demonstrate substantial reductions in extrapolation error and corroborate the theoretical predictions.
\end{abstract}



\section{Introduction}
Extrapolating a function beyond its sampled domain is a long-standing challenge in numerical analysis, with consequences for approximation theory, inverse problems, and modern regression \citep{engl1996regularization,hastie2009elements,trefethen2019approximation}. Classical constructions, including least squares (LS), regularized regression, and kernel-based methods~\cite{hastie2009elements, scholkopf2002learning, trefethen2019approximation}, are typically designed to minimize a fitting objective over the sampled region $\Omega$. While such objectives are well aligned with interpolation, they exert no direct control over the behavior of the resulting approximation in the extrapolation region $\Xi$. Consequently, extrapolation can be fundamentally unstable: small discrepancies on $\Omega$ can be strongly amplified on $\Xi$, even when the in-domain fit is excellent.

Recent theory makes this instability precise. In~\cite{hay2024function}, the amplification of sample-domain error into extrapolation error is quantified via an \emph{extrapolation condition number}, showing that minimizing empirical error on $\Omega$ does not imply meaningful performance on $\Xi$. This motivates formulations in which extrapolation is constrained directly in the target domain $\Xi$, rather than treated as an uncontrolled byproduct of an interpolation-oriented objective.

While~\cite{hay2024function} quantified extrapolation instability and proposed a neural-network correction heuristic, the present work develops a fully geometric feasibility-and-projection framework equipped with rigorous deterministic and probabilistic certification tools. Specifically, anchors are auxiliary approximations equipped with certified control in $\Xi$. From a collection of anchors and associated radii, we define a \emph{feasible set} of functions that remain bounded relative to the anchors on $\Xi$; see Definitions~\ref{definition:feasible_space} and~\ref{definition:full_feasible_space}. Given an arbitrary approximation $g$, regardless of how it is produced, we compute a corrected approximation $h$ by projecting $g$ onto a certified feasible ball, Definition~\ref{definition:feasible_space}. This projection defines a model-agnostic correction layer that depends only on certified control in $\Xi$. The analysis shows that the projection cannot increase the unknown extrapolation error on $\Xi$, and when $g$ lies outside the feasible set, it yields a quantifiable improvement with explicit bounds.

The perspective is closely aligned with localized reliability in trust-region methods for optimization~\cite{ConnGouldToint2000, NocedalWright2006}. In this analogy, anchors serve as trusted local models, while the feasible set serves as a functional trust region that restricts admissible behavior to the extrapolation domain $\Xi$, where stability is required. Unlike classical trust-region methods, where reliability is defined in parameter space, here the trust region is imposed directly in function space over a spatial domain.

A central question is how to construct anchor radii in a manner that is both rigorous and practically informative. Anchors may be available \emph{a priori}, for example, from physical constraints or trusted surrogate models, or they may be constructed \emph{a posteriori} from an existing approximation by certifying how in-domain error transfers to $\Xi$. To make this constructive, we develop certification tools that relate $\mathcal{E}_\Omega=\|f-g\|_{\Omega}$ to $\mathcal{E}_\Xi=\|f-g\|_{\Xi}$. The condition-number bound of~\cite{hay2024function}, summarized in Definition~\ref{def:classical_kappa} and Theorem~\ref{theorem:all_kappa_calculation}, provides such a certificate but can be overly conservative. We therefore derive a tight spectral bound based on the largest eigenvalue of the $\Xi$--Gram matrix formed from an $\Omega$--orthonormal basis; see Definition~\ref{def:spectral_kappa} and Theorem~\ref{theorem:all_kappa_calculation_improved}. This bound is tight by the Rayleigh-Ritz principle and is always sharper than the earlier estimate, as shown in Theorem~\ref{theorem:kappa_improved_always_better}. In addition, we introduce an inner-domain bound in Theorem~\ref{theorem:inner_kappa_calculation} that trades universality for improved numerical robustness, mitigating instabilities that may arise when evaluating bases or Gram matrices in extrapolative regimes.

While the spectral bound of Theorem~\ref{theorem:all_kappa_calculation_improved} is tight in a worst-case sense, it is governed by extremal directions that may not be representative of practical approximation errors. This motivates a second, probabilistic layer of certification. To better reflect typical behavior, we introduce a probabilistic refinement based on a directional error model. Conditioning on the in-domain magnitude $\|f-g\|_{\Omega}$, we model the normalized coefficient error direction in an $\Omega$--orthonormal basis as isotropic, uniform on the unit sphere, a standard hypothesis in directional statistics~\cite{mardia2000directional}. Under this model, the amplification $\|f-g\|_{\Xi}^2/\|f-g\|_{\Omega}^2$ is a random Rayleigh quotient. When the $\Xi$--Gram spectrum exhibits a dominant mode, the amplification is well approximated by a scaled Beta law, enabling explicit quantiles via classical distribution theory~\cite{johnson1995continuous}. The outcome is a family of high-confidence radii $\kappa_{\rho}$ such that
\[
\mathbb{P}\!\left(\|f-g\|_{\Xi}^2 \le \kappa_{\rho}\,\|f-g\|_{\Omega}^2\right)\ge \rho,
\]
with $\kappa_{\rho}$ typically substantially smaller than the eigenvalue-based constant. These probabilistic radii shrink feasible sets and strengthen the projection correction precisely in regimes where worst-case feasibility is too weak to be informative, consistent with high-dimensional concentration phenomena~\cite{vershyninHDP}.

This work builds on the anchor-based extrapolation framework of \cite{hay2024function}, but differs in three key respects. First, we introduce two extrapolation condition numbers; one is a tight spectral constant (attained by an extremal eigenpair of a $\Xi$-Gram operator), sharpening the generally non-tight bound in \cite{hay2024function}. Second, we replace the heuristic machine-learning correction of \cite{hay2024function} by an explicit $\Xi$-projection and prove a non-worsening extrapolation guarantee, together with computable lower and upper bounds on the improvement under feasibility; this projection step is model-agnostic and can be used as a post-processing layer on top of any baseline predictor. Third, we introduce a probabilistic anchor construction that yields a controllable radius for the feasible set, allowing one to tune the size of the admissible region (and hence the conservatism of the certificate) via a prescribed confidence level. In summary, we propose a geometric feasibility-and-projection framework for function extrapolation that separates interpolation from extrapolation and provides certified deterministic and probabilistic stability guarantees for the latter. By enforcing control directly on the extrapolation domain, the method provides a model-agnostic correction layer with rigorous performance bounds.

The contributions of this paper are:
\begin{enumerate}
    \item We establish a geometric feasible-set framework for extrapolation on $\Xi$ in which anchor functions generate a certified region that contains the unknown target and admits explicit size control.
    \item We prove that projecting an arbitrary approximation onto this region yields a model-agnostic correction layer that cannot increase the extrapolation error on $\Xi$ and provides explicit lower and upper bounds on the improvement whenever the original approximation is outside the feasible set.
    \item We develop rigorous tools for anchor creation and radius certification, including a Rayleigh-Ritz tight spectral bound and a stable inner-domain bound.
    \item We introduce a probabilistic extrapolation condition number based on the Rayleigh quotient of the $\Xi$–Gram matrix, yielding high-confidence feasible radii via explicit quantiles.
\end{enumerate}

\section{Problem formulation and prior work} \label{sec:problem_formulation}

We consider the problem of extrapolating a function $f$ from its samples and open the discussion with some required notation, followed by formulating the problem. 

\subsection{Problem formulation}

We study extrapolation of an unknown real-valued function from a \emph{sample domain} $\Omega$ to a (possibly disjoint) \emph{extrapolation domain} $\Xi$. We view both $\Omega$ and $\Xi$ as manifolds embedded in $\mathbb{R}^n$, and we consider a class $\mathcal{F}$ of candidate functions defined on $\Omega\cup\Xi$. On $\Xi$ we equip $\mathcal{F}$ with a norm $\|\cdot\|_{\Xi}$ induced by an inner product $\langle\cdot,\cdot\rangle_{\Xi}$. The ideal (but unattainable) extrapolation objective is therefore
\begin{equation}
g^{\ast}=\arg\min_{g\in \mathcal{F}} \norm{g-f}_\Xi ,
\label{eqn:ie_ideal_objective}
\end{equation}
where $f$ denotes the unknown target function.

For concreteness and to connect with standard approximation settings, we assume that $\mathcal{F}$ is finite-dimensional with a given basis $\{ \phi_k \}_{k=1}^d$. Hence any $g\in\mathcal{F}$ can be written as $g=\sum_{k=1}^d \beta_k \phi_k$, and in particular
\[
f=\sum_{k=1}^d \alpha_k \phi_k,
\]
with unknown coefficients $\alpha_k$, $k=1,\ldots,d$. Under this representation, estimating $f$ on $\Xi$ is equivalent to recovering (or accurately approximating) its coefficients from information on $\Omega$.

\begin{remark}[More general function spaces]
\label{remark:function_space}
We assume that $\mathcal{F}$ is a vector space in the above formulation. In some applications, however, additional prior information must be incorporated. For example, one may require functions to satisfy properties such as boundedness or monotonicity. Although these constraints can be imposed on elements of a linear space, the subsets they define are generally not closed under addition or scalar multiplication and therefore do not constitute linear subspaces. In this sense, such constraints introduce a nonlinear structure within $\mathcal{F}$, yielding a more restrictive class of admissible functions. We later show how these constraints can be incorporated into our method.  
\end{remark}

As noted in Remark~\ref{remark:function_space}, the proposed framework readily accommodates additional structural constraints. For clarity of presentation, we use the terms \emph{function space} and \emph{vector space} interchangeably throughout, unless explicitly stated otherwise.

In practice, $f$ is not observed directly; instead, we are given samples in $\Omega$ corrupted by noise.

\begin{definition}[Extrapolation problem] \label{def_extra_function_problem}
Let $f$ be a real-valued function defined over $\Omega$ and $\Xi$. We observe the data:
\begin{equation} \label{eqn:noisy_samples}
    y_i = f(x_i) + \varepsilon_i, 
    \quad x_i \in \Omega, \quad i = 0,\ldots,N ,
\end{equation}
where $\{\varepsilon_i\}_{i=0}^N$ are i.i.d. random variables. 
The problem is to estimate $f$ over $\Xi$ in the sense of~\eqref{eqn:ie_ideal_objective}.
\end{definition}

To separately quantify accuracy on the sampled region and on the extrapolation region, we introduce domain-specific norms and corresponding error measures.
\begin{definition}[Domain error and root error]
Let $f$ be the target function and $g$ an approximation, and let 
$D$ be a domain equipped with a norm $\|\cdot\|_D$.
We define the squared error of $g$ with respect to $f$ on $D$ by
\begin{equation}
  E_D(f,g) := \| f - g \|_D^2.
\end{equation}
When the reference function $f$ is clear from the context, we abbreviate
\[
E_D(g) := E_D(f,g).
\]
In addition, we use the corresponding unsquared (root) error
\begin{equation}
  \mathcal{E}_D(g) := \bigl( E_D(g) \bigr)^{1/2}
                     = \| f - g \|_D,
\end{equation}
and we write, in particular, $E_\Omega(g)$, $E_\Xi(g)$ and
$\mathcal{E}_\Omega(g)$, $\mathcal{E}_\Xi(g)$ for the sample and
extrapolation domains, respectively.
\label{eq_ie_mse_objective_2}
\end{definition}

Extrapolation is fundamentally under-determined from noisy samples on $\Omega$ alone. We therefore assume access to additional functions---\emph{anchors}---that are informative specifically on the extrapolation domain $\Xi$ (e.g., physics-based surrogates, coarse simulators, historical models, or bounds derived from analytic structure). The role of anchors is not to replace learning from data on $\Omega$, but to constrain or regularize behavior on $\Xi$.

We next introduce anchor functions. A related notion appears in \cite{hay2024function}; however, in the present work, anchors serve as the primary geometric mechanism underlying the feasibility-and-projection framework.
\begin{definition}[Anchor functions]
Let $\delta_j>0$, $j=1,\dots, M$ be a set of positive numbers. A set of functions $\{a_j\}_{j=1}^M$ are \(\delta_j\)-anchor functions for \(f\) on \(\Xi\) if,
\begin{equation} 
\mathcal{E}(a_j) = \|f-a_j\|_{\Xi}\le \delta_j,\quad j=1,\dots,M.
\nonumber
\end{equation}
\label{def_anchor_functions}
\end{definition}

With anchors in hand, we study extrapolation under both noisy observations on $\Omega$ and certified proximity information on $\Xi$.

\begin{definition}[Anchored extrapolation problem] \label{def_anchor_function_problem}
Extrapolate a function $f$ from noisy samples of~\eqref{eqn:noisy_samples} to $\Xi$ given a set of \(\delta_j\)-anchor functions.
\end{definition}
\begin{remark}[Certified anchor tolerance]\label{rem:certified-delta}
The anchor parameters $\delta_j$ should be interpreted as \emph{certified upper bounds} (possibly conservative) accompanying the anchor functions, rather than as oracle access to the unknown quantities $\|a_j - f\|_{\Xi}$. Throughout the paper, we assume that these parameters are chosen so that
\begin{equation}\label{eq:delta-certificate}
\mathcal{E}(a_j) \le \delta_j, \qquad j=1,\dots,M.
\end{equation}

Importantly, such certificates can be constructed from observable in-domain information. In Section~\ref{sec:extrapolation method} we derive extrapolation bounds that map an estimate of the approximation error on $\Omega$ to an explicit upper bound on the error in $\Xi$. These results provide a concrete data-driven procedure for selecting $\delta_j$ (or, in the probabilistic setting, choosing it according to a prescribed confidence level), thereby ensuring that~\eqref{eq:delta-certificate} holds without requiring access to $f$ on $\Xi$.

Choosing $\delta_j$ larger than necessary merely enlarges the feasible set, resulting in a more conservative correction. In contrast, underestimating $\delta_j$ may violate feasibility by excluding the true function from the admissible set.
\end{remark}

\section{A geometric feasibility-and-projection framework for extrapolation} \label{sec:extrapolation method}

We now introduce the core framework underlying our approach. 
Rather than attempting to directly minimize extrapolation error on $\Xi$, 
we reformulate extrapolation as a geometric feasibility problem in function space. 
Anchor functions, equipped with certified tolerance parameters, define a region of admissible behavior on $\Xi$ that is guaranteed to contain the true function. 
Given any baseline approximation, we obtain a corrected predictor by projecting it onto this certified feasible set. 
This perspective separates interpolation from extrapolation: the former is driven by data on $\Omega$, while the latter is controlled by geometric constraints on $\Xi$.

\subsection{Anchor function extrapolation}

Given anchor functions $\{a_i\}_{i=1}^{M}$ satisfying~\eqref{eq:delta-certificate}, we construct an extrapolation procedure for a noisy target function $f$ observed only on $\Omega$. 

Let $B=\{\phi_k\}_{k=1}^{d}$ be a basis spanning $\mathcal{F}=\mathrm{span}\{\phi_k\}_{k=1}^{d}$, and assume that $a_i \in \mathcal{F}$ for $i=1,\ldots,M$ and $f \in \mathcal{F}$.

\begin{definition}[anchor feasible set]
Given an anchor function $a \in \mathcal{F}$ and a tolerance parameter $\delta > 0$ satisfying $\|a - f\|_{\Xi} \le \delta$, we define the associated feasible set by
\[
\mathcal{S}(a,\delta) := \{ g \in \mathcal{F} \;:\; \|g - a\|_{\Xi} \le \delta \}.
\]
\label{definition:feasible_space}
\end{definition}
   

\begin{definition}[minimal feasible set]
Let $\{a_i\}_{i=1}^{M} \subset \mathcal{F}$ be anchor functions with associated tolerances $\{\delta_i\}_{i=1}^{M}$ satisfying $\|a_i - f\|_{\Xi} \le \delta_i$ for each $i$. 
The minimal feasible set is defined as the intersection of the individual anchor feasible sets:
\[
\mathcal{S} := \bigcap_{i=1}^{M} \mathcal{S}(a_i,\delta_i).
\]
\label{definition:full_feasible_space}
\end{definition}

We prove some general properties of the minimal feasible set, including an extrapolation error bound for all functions in the space. In later sections, we will introduce a method to choose a function within this space with a different proven error bound.
\begin{lemma} \label{lemma:ext_bound}
    Given anchor functions \( \{a_i\}_{i=1}^{M} \), with distance parameters \( \{\delta_i\}_{i=1}^{M} \) for \( f \), and minimal feasible set  \( s \) for all \( \{a_i\}_{i=1}^{M} \), the following hold:
    \begin{enumerate}
        \item $f$ is in the feasible set, that is \( f \in \mathcal{S} \).
        \item For any function $g\in \mathcal{S}$ it holds that: $\|f-g\|_\Xi \leq 2 \min{\{\delta_i\}_{i=1}^{M}}$.
    \end{enumerate}
    \label{lemma:search_space_size_upper_bound}
\end{lemma}
\begin{proof}
    We prove each claim separately.

\smallskip
\noindent\textit{Proof of (1).}
By assumption, $f \in \mathcal{F}$ and $a_i \in \mathcal{F}$ for all $i=1,\ldots,M$. 
From the definition of the feasible sets $\mathcal{S}_i$, we have $\|f-a_i\|_\Xi \le \delta_i$ for all $i$, and hence $f \in \mathcal{S}_i$ for each $i$. 
Since the minimal feasible set is defined by $\mathcal{S} = \bigcap_{i=1}^M \mathcal{S}_i$, it follows immediately that $f \in \mathcal{S}$.

\smallskip
\noindent\textit{Proof of (2).}
Without loss of generality, reorder the anchors so that 
\[
\delta_1 = \min_{1 \le i \le M} \delta_i.
\]
Let $g \in \mathcal{S}$. Then $g \in \mathcal{S}_1$, and by definition of $\mathcal{S}_1$,
\[
\|g-a_1\|_\Xi \le \delta_1.
\]
Using the triangle inequality,
\[
\|g-f\|_\Xi 
\le \|g-a_1\|_\Xi + \|a_1-f\|_\Xi 
\le \delta_1 + \delta_1 
= 2\delta_1.
\]
Since $\delta_1 = \min_{1 \le i \le M} \delta_i$, the desired bound follows.
\end{proof}
We note that when only one anchor function is present, the bound in Lemma~\ref{lemma:ext_bound} is tight. 

\subsection{Projection guarantees in the feasible ball and set}

Having defined the anchor feasible sets, we now analyze the effect of projecting an arbitrary approximation onto such a set. The results in this subsection are purely geometric: they require only that the true function $f$ lies inside the feasible region and do not depend on how the tolerance parameters were obtained. 

We show that projection onto a certified feasible ball cannot increase the extrapolation error on $\Xi$, and that whenever the original approximation lies outside the feasible set, the projection yields a strictly improved extrapolation estimate with explicit quantitative bounds.

The following lemma will be used in subsequent proofs and visualized in Figure~\ref{fig:lemma_vis}.
\begin{lemma}
Let $C$ be the circle of radius $R$ centered at the origin, let $P=(p,0)$ be a point with $p>R$ lying on the positive $x$--axis, let $B=(x,y)$ be any point with $x^2+y^2\le R^2$, and let $Q$ denote the orthogonal projection of $P$ onto $C$ (so $Q=(R,0)$). Define
\[
 d(B)=d(x,y)=|PB|-|QB|=\sqrt{(p-x)^2+y^2}-\sqrt{(R-x)^2+y^2}.
\]
Then $d$ attains a global maximum and a global minimum on the closed disk $\{x^2+y^2\le R^2\}$ and the sharp bounds
\[
 \frac{(p-R)^{3/2}}{\sqrt{p}}\ \le\ d(x,y)\ \le\ p-R
\]
hold for all such $B$. The upper bound $d=p-R$ is attained for every $B$ on the horizontal diameter $\{(x,0):|x|\le R\}$, and the lower bound is attained exactly at the two boundary points
\[
 x_\star=\frac{R(p+R)}{2p},\qquad y_\star=\pm\sqrt{R^2-x_\star^2},
\]
for which
\[
 d(x_\star,\pm y_\star)=\frac{(p-R)^{3/2}}{\sqrt{p}}.
\]
\label{lemma:circle_point_circle_with_bounds}
\end{lemma}
\begin{proof}
    Proof is in Appendix~\ref{appendix:lemma_circle_point_circle_with_bounds_proof}
\end{proof}
\begin{figure}[htbp]
    \centering
    \includegraphics[width=0.6\textwidth]{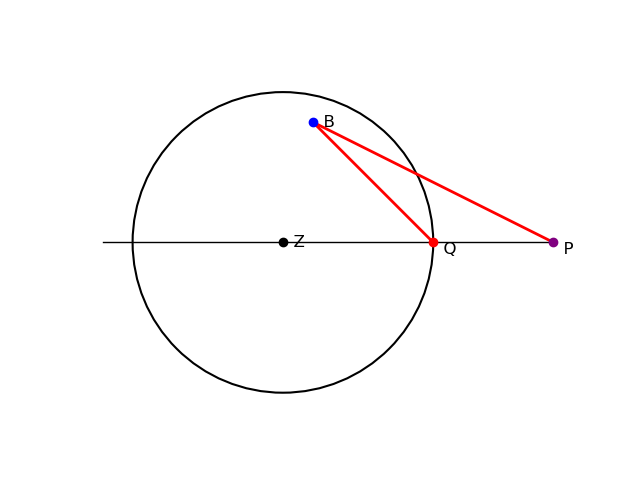}
    \caption{Geometric configuration used in Lemma~\ref{lemma:circle_point_circle_with_bounds} illustrating the projection of an exterior point onto a feasible ball and the resulting distance reduction.}
    \label{fig:lemma_vis}
\end{figure}

\begin{theorem} \label{thm:gurantee1}
    Let $a$ be an anchor function and let $\mathcal{S}$ be the feasible set of Definition~\ref{definition:feasible_space}, with radius $\delta \ge 0$. Let $g$ be an extrapolation approximation function, and let $h$ denote the metric projection of $g$ onto $\mathcal{S}(a,\delta)$. Then,
    \[
    \|f - h\|_\Xi \le \|f - g\|_\Xi.
    \]
    Moreover, if $g \notin \mathcal{S}(a,\delta)$ then
    \[
    \|f - h\|_\Xi < \|f - g\|_\Xi,
    \]
    and the improvement in the extrapolation error is followed by the bounds of Lemma~\ref{lemma:circle_point_circle_with_bounds}. Namely, denote $\Delta = \|g - h\|_\Xi$, the difference between the extrapolation error obeys
    \[
    \Delta \sqrt{\frac{\Delta}{\delta + \Delta}}
    \;\le\;
    \|f - g\|_\Xi - \|f - h\|_\Xi
    \;\le\;
    \Delta.
    \]
    \label{theorem:5}
\end{theorem}

\begin{proof}
If $g\in \mathcal{S}(a,\delta)$, then by definition of the minimization problem we have $h=g$, and the statement follows trivially. 
Assume henceforth that $g\notin \mathcal{S}(a,\delta)$.

Recall that $h$ is defined as a best approximation of $g$ over the feasible ball,
\[
h \in \arg\min_{u\in \mathcal{S}(a,\delta)} \|u-g\|_{\Xi}.
\]
Define the distance from $g$ to the feasible set by
\[
\Delta := \|h-g\|_{\Xi} = \inf_{u\in \mathcal{S}(a,\delta)} \|u-g\|_{\Xi}.
\]
Since $\mathcal{S}(a,\delta)=\{u\in\mathcal{F}:\|u-a\|_{\Xi}\le \delta\}$ is the closed $\Xi$-ball of radius $\delta$ centered at $a$ and $g$ lies outside it, the reverse triangle inequality implies that for any $u\in \mathcal{S}(a,\delta)$,
\[
\|g-u\|_{\Xi} \ge \|g-a\|_{\Xi}-\|u-a\|_{\Xi} \ge \|g-a\|_{\Xi}-\delta,
\]
hence
\[
\Delta = \|g-a\|_{\Xi}-\delta
\qquad\text{and therefore}\qquad
\|g-a\|_{\Xi}=\delta+\Delta.
\]

Consider the affine subspace $\mathrm{aff}\{a,f,g\}$ spanned by the three points $a,f,g$, which has dimension at most two. Restricting all points and distances to this subspace does not change any of the norms appearing in the claim. 
We may therefore apply Lemma~\ref{lemma:circle_point_circle_with_bounds} to the configuration consisting of the $\Xi$-ball of radius $\delta$ centered at $a$, an interior point $f$ satisfying $\|f-a\|_{\Xi}\le \delta$, and an exterior point $g$ satisfying $\|g-a\|_{\Xi}=\delta+\Delta$, with $h$ being the closest point to $g$ on the ball.

Lemma~\ref{lemma:circle_point_circle_with_bounds} then yields the strict improvement
\[
\|f-h\|_{\Xi} < \|f-g\|_{\Xi},
\]
and moreover provides the quantitative bounds
\[
\Delta \sqrt{\frac{\Delta}{\delta+\Delta}}
\;\le\;
\|f-g\|_{\Xi}-\|f-h\|_{\Xi}
\;\le\;
\Delta.
\]
This shows that the reduction in extrapolation error is always at most the distance from $g$ to the feasible set, and is bounded below by the same distance times the factor $\sqrt{\Delta/(\delta+\Delta)}$.
\end{proof}

Theorem~\ref{thm:gurantee1} turns anchor certification into a genuine extrapolation guarantee. Under feasibility, the correction step directly controls the quantity of interest, namely the error on $\Xi$ with respect to the unknown target $f$. This stands in contrast to approaches that improve fit on the observation domain $\Omega$ and then appeal to a condition-number bound to infer performance on $\Xi$; when the extrapolation condition number is large, such bounds can be too loose for reductions on $\Omega$ to translate into meaningful improvements on $\Xi$. 

Here, the projection is performed in the $\Xi$-norm itself, so the update constitutes a certified non-worsening (and, when $g\notin S$, strictly improving) post-processing step for extrapolation rather than an indirect consequence of $\Omega$-domain accuracy. Moreover, the theorem provides computable lower and upper bounds on the improvement, quantifying the possible reduction in $\|g-f\|_{\Xi}$ in terms of the distance of the baseline predictor to the certified feasible set.

The projection result established above was derived for a single feasible ball. However, in practice, the minimal feasible set is given by the intersection of multiple anchor balls. Since each anchor feasible set is a closed convex subset of the Hilbert space $(F,\|\cdot\|_{\Xi})$, their intersection is also closed and convex. This allows us to extend the non-worsening projection guarantee to the full multi-anchor feasible set.
\begin{theorem}[Projection onto convex feasible sets]\label{thm:convex_projection}
Let $\mathcal S = \bigcap_{i=1}^M \mathcal S_i$ be the minimal feasible set, where each
\[
\mathcal S_i = \{h\in F:\|h-a_i\|_{\Xi}\le \delta_i\}
\]
is a closed $\Xi$--ball. Then $\mathcal S$ is closed and convex. In addition, let $g\in F$ and let $h$ be the $\Xi$--metric projection of $g$ onto $\mathcal S$, i.e.,
\[
h\in\arg\min_{u\in\mathcal S}\|u-g\|_{\Xi}.
\]
If $f\in\mathcal S$, then
\[
\|f-h\|_{\Xi} \le \|f-g\|_{\Xi}.
\]
Moreover, if $g\notin\mathcal S$, the inequality is strict.
\end{theorem}

\begin{proof}
Since each $\mathcal S_i$ is a closed convex subset of the Hilbert space $(F,\|\cdot\|_{\Xi})$, their intersection $\mathcal S$ is also closed and convex. The metric projection onto a closed convex set in a Hilbert space exists and is unique.

By definition,
\[
\|f-g\|_{\Xi}^2  = \|f-h+h-g\|_{\Xi}^2 = \|f-h\|_{\Xi}^2+ \|h-g\|_{\Xi}^2 +2 \langle f-h, h-g\rangle_{\Xi}
\]
and since $\langle g-h, f-h\rangle_{\Xi} \le 0$, as $f \in \mathcal S$, we get that $\langle f-h, h-g\rangle_{\Xi} \ge 0$ and so
\[
\|f-g\|_{\Xi}^2  \ge  \|f-h\|_{\Xi}^2+ \|h-g\|_{\Xi}^2  .
\]
Therefore
\[
\|f-h\|_{\Xi}
\le \|f-g\|_{\Xi}  ,
\]
which becomes strict if $g \notin \mathcal S$. This proves the claim. 
\end{proof}

\subsection{Extrapolation bounds and condition numbers}
\label{section:creating_anchor_functions}

The projection guarantees established in the previous subsection require that the true function $f$ lies within a certified feasible set. We now address how such certification can be obtained. Building on the extrapolation condition number introduced in \cite{hay2024function}, we develop quantitative bounds that relate in-domain error on $\Omega$ to extrapolation error on $\Xi$. 

These bounds provide explicit tolerance parameters for anchor functions and form the analytical foundation for the feasible regions used in the projection step.

\begin{definition}[The extrapolation condition number from \cite{hay2024function}]\label{def:classical_kappa}
Let the set $\{\varphi_k\}_{k=1}^d$ be an orthogonal basis on $\Omega$. Define
\[
M_\Xi := \max_{k=1,\dots,d} \|\varphi_k\|_\Xi^2,\qquad 
m_\Omega := \min_{k=1,\dots,d} \|\varphi_k\|_\Omega^2,
\]
and set
\[
\kappa := \frac{d\,M_\Xi}{m_\Omega}.
\]
\end{definition}

\begin{theorem}
    Let $\{\phi_k\}_{k=1}^d$ be a sequence of real-valued functions, defined both in $\Omega$ and $\Xi$. The following conditions hold:
    \begin{enumerate}
        \item If $\phi_k$ are orthogonal in $\Omega$, then $\mathcal{E}_{\Xi} \leq \sqrt{\kappa} \mathcal{E}_\Omega$.
        \item If $\phi_k$ are orthogonal both in $\Omega$ and $\Xi$, then $\mathcal{E}_{\Xi} \leq \sqrt{\frac{\kappa}{d}} \mathcal{E}_\Omega$.
    \end{enumerate}
    \label{theorem:all_kappa_calculation}
\end{theorem}

As illustrated in Figure~\ref{fig:bound_comparison}, the bound in Theorem~\ref{theorem:all_kappa_calculation} can be overly conservative, often producing very large extrapolation constants. 
To remedy this, we introduce an alternative extrapolation bound that is provably tighter under the same assumptions. 
Moreover, unlike the earlier estimate, the new bound is sharp.

\begin{definition}[Spectral extrapolation condition number]
    \label{def:spectral_kappa}
    Let $\{\tilde{\phi}_k\}_{k=1}^d$ be a sequence of real-valued functions
    defined on $\Omega$ and $\Xi$, and assume that
    $\{\tilde{\phi}_k\}_{k=1}^d$ are orthonormal on $\Omega$.
    Define the $\Xi$–Gram matrix
    \[
        \tilde G \;=\; \big[\langle \tilde{\phi}_i,\tilde{\phi}_j\rangle_{\Xi}\big]_{i,j=1}^d ,
    \]
    and set
    \[
        \kappa_{\mathrm{spec}} := \lambda_{\max}(\tilde G) ,
    \]
    where $\lambda_{\max}(\tilde G)$ is the largest eigenvalue of $\tilde G$.
    We call $\kappa_{\mathrm{spec}}$ the \emph{spectral extrapolation condition number}.
\end{definition}

\begin{theorem}[Spectral extrapolation bound]
    \label{theorem:all_kappa_calculation_improved}
    Let $\{\tilde{\phi}_k\}_{k=1}^d$ and $\kappa_{\mathrm{spec}}$ be as in
    Definition~\ref{def:spectral_kappa}. Then
    \[
        \mathcal{E}_{\Xi} \leq \sqrt{\kappa_{\mathrm{spec}}}\, \mathcal{E}_\Omega .
    \]
    Equivalently,
    \[
        \mathcal{E}_{\Xi} \leq \sqrt{\lambda_{\max}(\tilde G)}\, \mathcal{E}_\Omega .
    \]
    Moreover, this bound is tight.
\end{theorem}

\begin{proof}
    Given functions $f,g$ we define $e=f-g$. Expand $e$ in the $\Omega$–orthonormal basis: $e=\sum_{k=1}^d c_k\,\tilde\phi_k$ and let $c=(c_1,\ldots,c_d)^\top$.
By orthonormality on $\Omega$,
\[
E_{\Omega}=\|e\|_{\Omega}^2=\sum_{k=1}^d c_k^2=\|c\|_2^2.
\]

By definition of $\tilde G$,
\[
E_{\Xi}=\|e\|_{\Xi}^2
= \bigg\langle \sum_i c_i\tilde\phi_i,\sum_j c_j\tilde \phi_j \bigg\rangle_{\Xi}
= \sum_{i,j} c_i c_j \,\langle \tilde\phi_i,\tilde\phi_j\rangle_{\Xi}
= c^\top \tilde G c.
\]
Since $\tilde G$ is symmetric positive semidefinite, the Rayleigh–Ritz principle yields
\[
E_{\Xi}=c^\top \tilde G c \;\le\; \lambda_{\max}(\tilde G)\,\|c\|_2^2
= \lambda_{\max}(\tilde G)\,E_{\Omega}  .\]
Taking the square roots yields the required inequality,
$\mathcal{E}_\Xi \leq \sqrt{\lambda_{\max}(\tilde G)}\mathcal{E}_\Omega.$

Since the Rayleigh-Ritz principle states that the bound is tight, we get that the overall bound is tight as well.
\end{proof}

Since Theorem~\ref{theorem:all_kappa_calculation_improved} and Theorem~\ref{theorem:all_kappa_calculation} are proven on orthonormal and orthogonal bases, respectively, it is not immediately clear which is preferred. Therefore, we prove in Theorem~\ref{theorem:kappa_improved_always_better} that Theorem~\ref{theorem:all_kappa_calculation_improved} is always tighter than Theorem~\ref{theorem:all_kappa_calculation} and also show an empirical study in Figure~\ref{fig:bound_comparison}.

\begin{theorem}
\label{theorem:kappa_improved_always_better}
Let $\{{\phi}_k\}_{k=1}^d$ be a family of real-valued functions defined on both
$\Omega$ and $\Xi$ and orthogonal on $\Omega$, and let $\{\tilde{\phi}_k\}_{k=1}^d$ be the normalized
version such that $\{\tilde{\phi}_k\}_{k=1}^d$ is orthonormal on $\Omega$.
Define the $\Xi$–Gram matrix
\[
    \tilde G \;=\; \big[\langle \tilde{\phi}_i, \tilde{\phi}_j \rangle_{\Xi}\big]_{i,j=1}^d .
\]
Let $\kappa_{\mathrm{spec}}$ be the spectral extrapolation condition number
from Definition~\ref{def:spectral_kappa}, and let
$\kappa$ be the extrapolation condition number
from Definition~\ref{def:classical_kappa}, i.e.
$\kappa = d\, M_{\Xi} / m_{\Omega}$, where $M_{\Xi}$ and
$m_{\Omega}$ are defined using $\{{\phi}_k\}_{k=1}^d$.
Then
\[
    \kappa_{\mathrm{spec}} \;\leq\; \kappa.
\]
Equivalently,
\[
    \lambda_{\max}(\tilde G) \;\leq\; d \,\frac{M_{\Xi}}{m_\Omega},
\]
where $\lambda_{\max}(\tilde G)$ denotes the largest eigenvalue of $\tilde G$.
\end{theorem}

\begin{proof}
Let $\{\phi_k\}_{k=1}^d$ be real-valued functions defined on both $\Omega$ and $\Xi$, and define
\[
\tilde{\phi}_k = \frac{\phi_k}{\|\phi_k\|_\Omega}, \qquad k=1,\ldots,d,
\]
so that $\{\tilde{\phi}_k\}_{k=1}^d$ form an orthonormal system on $\Omega$.  
Define the $\Xi$–Gram matrix
\[
\tilde{G} = \big[\langle \tilde{\phi}_i, \tilde{\phi}_j \rangle_{\Xi}\big]_{i,j=1}^d .
\]

Since $\|\phi_k\|_{\Xi}^2 \leq M_{\Xi}$ by definition, we obtain
\[
|\langle \phi_i, \phi_j \rangle_{\Xi}| \leq M_{\Xi}.
\]
Hence, for the normalized functions,
\[
|\tilde{G}_{ij}| = \frac{|\langle \phi_i, \phi_j \rangle_{\Xi}|}{\|\phi_i\|_\Omega \, \|\phi_j\|_\Omega}
\le \frac{M_{\Xi}}{\|\phi_i\|_\Omega \, \|\phi_j\|_\Omega}.
\]
Because $m_{\Omega} = \min_k \|\phi_k\|_{\Omega}^2$, it follows that $\|\phi_k\|_\Omega \geq \sqrt{m_\Omega}$ for all $k$, and thus
\[
|\tilde{G}_{ij}| \leq \frac{M_{\Xi}}{m_{\Omega}} \qquad \forall i,j.
\]

Every row and column of $\tilde{G}$ then satisfies
\[
\sum_{j=1}^d |\tilde{G}_{ij}| \leq d \frac{M_{\Xi}}{m_{\Omega}},
\]
so that the matrix $1$–norm and $\infty$–norm satisfy $\|\tilde{G}\|_{1} \le d\frac{M_{\Xi}}{m_{\Omega}}$ and $\|\tilde{G}\|_{\infty} \le d\frac{M_{\Xi}}{m_{\Omega}}$.  
By the standard matrix norm inequality~\citep[Theorem~5.6.9]{horn2012matrix},
\[
\|\tilde{G}\|_{2} \leq \sqrt{\|\tilde{G}\|_{1}\,\|\tilde{G}\|_{\infty}}.
\]
Since $\tilde{G}$ is symmetric positive semidefinite, $\|\tilde{G}\|_{2} = \lambda_{\max}(\tilde{G})$, and therefore
\[
\lambda_{\max}(\tilde{G}) \leq  d\,\frac{M_{\Xi}}{m_{\Omega}}.
\]

\end{proof}

Although Theorem~\ref{theorem:all_kappa_calculation_improved} is tighter than Theorem~\ref{theorem:all_kappa_calculation}, taking the eigenvalue of the Gram matrix and the Gram matrix itself can be numerically unstable. To address this, we propose an alternative extrapolation bound, using a new condition ratio $\kappa_r$. The resulting bound is usually tighter than the one provided in Theorem~\ref{theorem:all_kappa_calculation}, but it is only applicable in scenarios where a given condition is met. In the next section, we provide a numerical comparison between the bounds in Figure~\ref{fig:bound_comparison_with_inner}.

\begin{theorem} 
    Let $\{\hat{\phi}_i\}_{i=1}^{d}$ be a set of real-valued functions defined on both $\Omega$ and $\Xi$, with $\Omega \cup \Xi = A$. Assume the functions are orthonormal on $A$, and that 
    \begin{equation} \label{eqn:cond_new_kappa}
        \hat{M}_{\Xi}= \max_{k=1,\ldots, d}\norm{\phi_{k}}^2_{\Xi} \leq \frac{1}{d} .
    \end{equation}
    Then, 
    \[
    \mathcal{E}_\Xi \leq \sqrt{\kappa_{r}} \, \mathcal{E}_\Omega, \quad \kappa_{r} = \frac{d \hat{M}_\Xi}{1 - d \hat M_\Xi}.
    \]
    \label{theorem:inner_kappa_calculation}
\end{theorem}

\begin{proof}
    Set $\phi_k := \hat{\phi}_k$ and note that by assumption the family $\{\phi_k\}_{k=1}^d$ is orthonormal on $A = \Omega \cup \Xi$. In particular, it is orthogonal on $A$.

    Theorem~\ref{theorem:all_kappa_calculation} does not impose any geometric or measure-theoretic relationship between its two domains (denoted there by $\Omega$ and $\Xi$); it only requires that the functions be orthogonal on the first domain. Therefore, we may apply it with
    \[
        \Omega := A, \qquad \Xi := \Xi,
    \]
    and the same system $\{\phi_k\}_{k=1}^d$.

    In the proof of Theorem~\ref{theorem:all_kappa_calculation} (case~1), the constant $\kappa$ is given by
    \[
        \kappa \;=\; \sum_{k=1}^d \frac{\|\phi_k\|_{\Xi}^2}{\|\phi_k\|_{A}^2},
    \]
    when $\{\phi_k\}$ is orthonormal on the first domain (here, $A$), i.e., $\|\phi_k\|_A^2 = 1$ for all $k$. Hence,
    \[
        \kappa = \sum_{k=1}^d \|\phi_k\|_{\Xi}^2 
        \;\leq\; d \max_{k} \|\phi_k\|_{\Xi}^2
        \;=\; d \hat{M}_\Xi.
    \]
    Applying part~(1) of Theorem~\ref{theorem:all_kappa_calculation} with $\Omega = A$ gives
    \[
        E_\Xi \leq \kappa E_A \leq d \hat{M}_\Xi E_A.
    \]
    Since $A = \Omega \cup \Xi$ and we measure the same squared error on each region, we have
    \[
        E_A = E_\Omega + E_\Xi.
    \]
    Substituting this into the previous inequality yields
    \[
        E_\Xi \leq d \hat{M}_\Xi (E_\Omega + E_\Xi).
    \]
    Rearranging,
    \[
        E_\Xi - d \hat{M}_\Xi E_\Xi \leq d \hat{M}_\Xi E_\Omega
        \quad\Longrightarrow\quad
        (1 - d \hat{M}_\Xi) E_\Xi \leq d \hat{M}_\Xi E_\Omega.
    \]
    By the assumption $\hat{M}_\Xi \leq 1/d$, we have $1 - d \hat{M}_\Xi \ge 0$; in the nontrivial case $1 - d \hat{M}_\Xi > 0$ we may divide both sides to obtain
    \[
        E_\Xi \leq \frac{d \hat{M}_\Xi}{1 - d \hat{M}_\Xi} E_\Omega \quad\Longrightarrow\quad \mathcal{E}_\Xi \leq \sqrt{\kappa_r} \mathcal{E}_\Omega.
    \]
\end{proof}

\begin{remark}
    The difference between $M_\Xi$ and $\tilde M_\Xi$ stems from the basis functions used in the calculation:
    \begin{enumerate}[label=(\roman*), leftmargin=*]        
    \item $\tilde M_\Xi$ is for $\Tilde{\phi}$ when they are orthonormal over $\Omega$.
    \item $\hat M_\Xi$ is for $\phi$ when they are orthonormal over $A$.
    \end{enumerate}

    Even if the basis is orthonormal on $\Omega$, the Gram matrix on $\Xi$ can be ill-conditioned or have a large top eigenvalue if the basis functions become highly correlated or large in magnitude on $\Xi$. This often results in a very large value, causing Theorem~\ref{theorem:all_kappa_calculation} to produce bounds that are too large to be useful.

    On the other hand, Theorem~\ref{theorem:inner_kappa_calculation} may not be applicable in many cases, but usually provides a tighter bound when it is valid. Note that an empirical comparison between the bounds from Theorem~\ref{theorem:all_kappa_calculation} and Theorem~\ref{theorem:inner_kappa_calculation} is presented in a later section on Figure~\ref{fig:bound_comparison_with_inner}.
    \label{remark:bound_comparison_with_inner}
\end{remark}

\subsection{Probabilistic modeling of extrapolation amplification}
\label{subsec:prob_anchor_sphere}

The extrapolation bounds in Theorems~\ref{theorem:all_kappa_calculation}, 
\ref{theorem:all_kappa_calculation_improved}, and~\ref{theorem:inner_kappa_calculation} 
are derived in a worst-case sense. While fully rigorous, these bounds are governed by extremal spectral behavior and may therefore be overly conservative in practice. As a consequence, the resulting feasible sets can be unnecessarily large, providing limited practical restrictions. Empirical comparisons later in the paper (see Figure~\ref{fig:bound_comparison}) illustrate this conservatism. This observation motivates a probabilistic refinement of the extrapolation constant.

To obtain bounds that better reflect typical behavior, we also consider a probabilistic model. Writing $f-g$ in an $\Omega$-orthonormal basis with coefficient vector $c$, we condition on its (observed or estimated) magnitude $\mathrm{m}=\|c\|_2$ and model only the normalized direction $c/\|c\|_2$ as random. For simplicity, we assume this direction is uniform on the sphere, i.e., $c$ is uniformly distributed over
\[
\{x\in\mathbb{R}^d:\|x\|_2=\|c\|_2\}.
\]
This isotropic assumption is a convenient baseline when no additional structure on the error direction is available.

Under the above modeling, we derive bounds that hold with a prescribed confidence level and are typically much smaller than their worst-case counterparts. The following lemma, corollary, and definitions establish these probabilistic bounds.

\begin{lemma}[Sphere quadratic form for a pure dominant-eigenvalue matrix]
\label{lem:rank_one_beta_sphere}
Let $d\ge 2$ and let $s\in\mathbb{R}^d$ be uniformly distributed on the unit sphere $S^{d-1}=\{x\in\mathbb{R}^d:\|x\|_2=1\}$. Let $u\in\mathbb{R}^d$ be a unit vector and set $G:=\lambda_{\max}uu^\top$ with $\lambda_{\max}>0$. Then
\[
\frac{s^\top G s}{\lambda_{\max}}=(u^\top s)^2
\;\overset{d}{=}\; s_1^2
\;\sim\; \mathrm{Beta}\!\left(\tfrac12,\tfrac{d-1}{2}\right),
\]
where $\overset{d}{=}$ denotes equality in distribution.
In particular, if $c:=\mathrm{m}s$ for some $\mathrm{m}>0$ (so $\|c\|_2=\mathrm{m}$), then
\[
c^\top G c \;\overset{d}{=}\; \lambda_{\max}\,\|c\|_2^2\,Z,
\qquad
Z\sim \mathrm{Beta}\!\left(\tfrac12,\tfrac{d-1}{2}\right).
\]
\end{lemma}

\begin{proof}
By rotational invariance of the uniform distribution on $S^{d-1}$ we may assume $u=e_1$,
so $s^\top G s=\lambda_{\max}s_1^2$.

Let $X=(X_1,\ldots,X_d)\sim\mathcal{N}(0,I_d)$. Since $X_1$ is independent of $(X_2,\ldots,X_d)$,
the random variables
\[
U:=X_1^2\sim\chi^2_1
\quad\text{and}\quad
V:=\sum_{i=2}^d X_i^2\sim\chi^2_{d-1}
\]
are independent, and hence the classical identity
\[
\frac{U}{U+V}\sim \mathrm{Beta}\!\left(\tfrac12,\tfrac{d-1}{2}\right)
\]
applies \cite{johnson1995continuous}.

Next, by the standard normalization construction of a uniform random direction,
\[
s \;\overset{d}{=}\; \frac{X}{\|X\|_2},
\]
see, e.g., \cite{mardia2000directional}. Therefore
\[
s_1^2=\frac{X_1^2}{\sum_{i=1}^d X_i^2}=\frac{U}{U+V}
\sim \mathrm{Beta}\!\left(\tfrac12,\tfrac{d-1}{2}\right),
\]
which proves the first claim. For $c=\mathrm{m}s$ we have
$c^\top G c=\mathrm{m}^2 s^\top G s$ and $\|c\|_2=\mathrm{m}$, yielding the scaling statement.
\end{proof}

\begin{corollary}[dominant eigenvalue approximation on the sphere]
\label{cor:almost_rank_one_beta_sphere}
Let $G$ be positive semidefinite with eigenvalues
$\lambda_1=\lambda_{\max}\ge \lambda_2\ge\cdots\ge \lambda_d\ge 0$ and spectral decomposition
$G=U\Lambda U^\top$.
Let $\mathcal{S}$ be uniform on $S^{d-1}$ and set $c=\mathrm{m}s$ with $\mathrm{m}>0$.
Then
\[
c^\top G c
= \mathrm{m}^2 \sum_{i=1}^d \lambda_i\,Y_i^2,
\qquad
Y:=U^\top s \ \text{(hence also uniform on $S^{d-1}$)},
\]
and consequently $Y_1^2\sim \mathrm{Beta}\!\left(\tfrac12,\tfrac{d-1}{2}\right)$.
Moreover,
\[
\mathrm{m}^2\lambda_1 Y_1^2
\;\le\;
c^\top G c
\;\le\;
\mathrm{m}^2\!\left(\lambda_1 Y_1^2 + \lambda_2(1-Y_1^2)\right).
\]
In particular, when $\lambda_2\ll \lambda_1$ (strong spectral gap),
the distribution of $c^\top G c$ is well-approximated by
\[
c^\top G c \;\approx\; \lambda_{\max}\,\|c\|_2^2\,Z,
\qquad
Z\sim \mathrm{Beta}\!\left(\tfrac12,\tfrac{d-1}{2}\right).
\]
\end{corollary}

\begin{remark}[Consistency with the mean-eigenvalue identity]
\label{rem:mean_eigenvalue_consistency}
A classical isotropy property of the uniform distribution on the sphere states that if
$s\sim \mathrm{Unif}(S^{d-1})$, then
\begin{equation}
\label{eq:isotropy_sphere}
\mathbb{E}[ss^\top] \;=\; \frac{1}{d}I_d ,
\end{equation}
see, e.g., \cite{mardia2000directional}. Consequently, for any symmetric positive semidefinite
matrix $G$ with eigenvalues $\lambda_1,\ldots,\lambda_d$,
\begin{equation}
\label{eq:quadratic_form_mean_trace}
\mathbb{E}\!\left[s^\top G s\right]
\;=\;
\mathrm{tr}\!\bigl(G\,\mathbb{E}[ss^\top]\bigr)
\;=\;
\frac{1}{d}\mathrm{tr}(G)
\;=\;
\frac{1}{d}\sum_{i=1}^d \lambda_i .
\end{equation}
If $c=\mathrm{m}\,s$ (so $\|c\|_2=\mathrm{m}$), then
\begin{equation}
\label{eq:quadratic_form_mean_trace_scaled}
\mathbb{E}\!\left[c^\top G c\right]
\;=\;
\|c\|_2^2\,\frac{\mathrm{tr}(G)}{d}
\;=\;
\mathrm{m}^2\,\frac{1}{d}\sum_{i=1}^d \lambda_i .
\end{equation}

In the dominant eigenvalue setting of Lemma~\ref{lem:rank_one_beta_sphere},
$c^\top G c \overset{d}{=} \lambda_{\max}\|c\|_2^2 Z$ with
$Z\sim\mathrm{Beta}(\tfrac12,\tfrac{d-1}{2})$.
Since $\mathbb{E}[Z]=\frac{a}{a+b}=\frac{1/2}{1/2+(d-1)/2}=\frac{1}{d}$, we obtain
\[
\mathbb{E}[c^\top G c]
=
\lambda_{\max}\|c\|_2^2\,\mathbb{E}[Z]
=
\lambda_{\max}\|c\|_2^2\,\frac{1}{d},
\]
which agrees with~\eqref{eq:quadratic_form_mean_trace_scaled} because
$\mathrm{tr}(G)\approx\lambda_{\max}$ when $G$ has a dominant eigenvalue.
\end{remark}

The preceding analysis characterizes the distribution of the
extrapolation amplification ratio under the isotropic direction
model. In particular, it identifies the Rayleigh quotient
$s^\top \tilde G s$ as the fundamental random quantity governing
out-of-domain error growth. We now leverage this characterization
to construct high-confidence extrapolation bounds based on
quantiles of this distribution.

\subsection{High-confidence probabilistic extrapolation bound}

To obtain high-confidence bounds, we introduce a probabilistic model
for the coefficient error. Writing $f-g$ in an $\Omega$--orthonormal
basis with coefficient vector $c$, we condition on its magnitude
$m=\|c\|_2$ and model only its normalized direction $c/\|c\|_2$.
We assume this direction is uniformly distributed on the unit sphere
$S^{d-1}\subset\mathbb{R}^d$.

\begin{definition}[Probabilistic spectral extrapolation condition number]
\label{def:prob_spectral_kappa_sphere}
Let $G$ be the Gram matrix from Definition~\ref{def:spectral_kappa}.
Let $\mathcal{S}$ be uniformly distributed on the unit sphere $S^{d-1}\subset\mathbb{R}^d$ and define
\[
\kappa_{\mathrm{spec}}^{\mathrm{prob}} := s^\top G s.
\]
Equivalently, for any $\mathrm{m}>0$ and $c=\mathrm{m}s$ (so $\|c\|_2=\mathrm{m}$),
\[
\kappa_{\mathrm{spec}}^{\mathrm{prob}}=\frac{c^\top G c}{\|c\|_2^2}.
\]
In the dominant eigenvalue case $G=\lambda_{\max}uu^\top$ with $\|u\|_2=1$,
Lemma~\ref{lem:rank_one_beta_sphere} implies
\[
\kappa_{\mathrm{spec}}^{\mathrm{prob}}
\sim \lambda_{\max} Z,
\qquad
Z\sim \mathrm{Beta}\!\left(\tfrac12,\tfrac{d-1}{2}\right),
\]
and when $\lambda_{\max}\gg\lambda_2,\dots,\lambda_d$ this gives a good approximation.
\end{definition}

\begin{definition}[Quantile of the probabilistic condition number]
\label{def:prob_anchor_function_sphere}
For $\rho\in(0,1)$, define
\[
\kappa^{\rho}_{\mathrm{spec}}
:=
\inf\Big\{
t\in\mathbb{R}
:
\mathbb{P}\big(
\kappa^{\mathrm{prob}}_{\mathrm{spec}}
\le t
\big)
\ge \rho
\Big\}.
\]
\end{definition}

To observe the following bound, we write the coefficient error $e=f-g$ in the
$\Omega$--orthonormal basis as,
\[
e = \sum_{k=1}^d c_k \tilde\phi_k,
\qquad c = (c_1,\dots,c_d)^\top.
\]
We then have,
\[
E_\Omega(f,g) = \|e\|_\Omega^2 = \|c\|_2^2,
\qquad
E_\Xi(f,g) = \|e\|_\Xi^2 = c^\top \tilde G c.
\]
If we write $c = \|c\|_2 s$ with
$s = c/\|c\|_2 \in S^{d-1}$, then
\[
\frac{E_\Xi(f,g)}{E_\Omega(f,g)}
=
\frac{c^\top \tilde G c}{\|c\|_2^2}
=
s^\top \tilde G s.
\]
Under the isotropic direction model, this ratio is precisely the random variable $\kappa^{\mathrm{prob}}_{\mathrm{spec}}$. Therefore, we arrive at:

\begin{proposition}[Probabilistic spectral extrapolation bound]
\label{prop:prob_spec_bound}
Fix $\rho\in(0,1)$ and let
$\kappa^{\rho}_{\mathrm{spec}}$
be the $\rho$--upper quantile of
$\kappa^{\mathrm{prob}}_{\mathrm{spec}}$.
Then for any $g\in F$,
\begin{equation} \label{eq:prob_bound_root_sphere} 
\mathbb{P}\!\left(
E_\Xi(f,g)
\le
\sqrt{\kappa^{\rho}_{\mathrm{spec}}}\,
E_\Omega(f,g)
\right)
\ge \rho.
\end{equation}
\end{proposition}
\begin{proof}
The claim follows directly from the definition of the
$\rho$--upper quantile of
$\kappa^{\mathrm{prob}}_{\mathrm{spec}}$
and the identity
$E_\Xi(f,g)
=
\kappa^{\mathrm{prob}}_{\mathrm{spec}} E_\Omega(f,g)$.
\end{proof}


\begin{remark}[Why we do not rely on generic sphere concentration]
\label{rem:concentration_vs_beta}
Let $s\sim\mathrm{Unif}(S^{d-1})$ and let $G\succeq 0$ with
$\lambda_{\max}=\|G\|_{\mathrm{op}}$.
Generic concentration on the sphere (e.g.\ L\'evy's lemma; see \cite[Thm.~5.1.3]{vershyninHDP})
applied to the quadratic form $f(s):=s^\top G s$ yields the quantile bound
\begin{equation}
\label{eq:generic_concentration_quantile_trace_main}
Q_{1-\delta}\!\bigl(s^\top G s\bigr)
\;\le\;
\frac{\mathrm{tr}(G)}{d}
+
C\,\lambda_{\max}\sqrt{\frac{\log(2/\delta)}{d}}
\end{equation}
for an absolute constant $C>0$; see Appendix~\ref{app:derivation_generic_sphere_quantile}.

In contrast, in the dominant eigenvalue case $G=\lambda_{\max}uu^\top$,
Lemma~\ref{lem:rank_one_beta_sphere} gives the \emph{exact} distribution
$s^\top G s=\lambda_{\max}Z_d$ with $Z_d\sim\mathrm{Beta}(\tfrac12,\tfrac{d-1}{2})$.
For any fixed confidence level $\rho$ (e.g.\ $0.95$), one has
$Q_{\rho}(Z_d)=\Theta(1/d)$, hence
$Q_{\rho}(s^\top G s)=\Theta(\lambda_{\max}/d)$.
By comparison, the dominant term in~\eqref{eq:generic_concentration_quantile_trace_main}
is $\Theta(\lambda_{\max}/\sqrt d)$ for fixed $\rho$, i.e.\ looser by a factor
$\Theta(\sqrt d)$.
This is why we prefer the distributional (dominant eigenvalue) model when deriving a tight probabilistic extrapolation condition number.
\end{remark}

The probabilistic analysis above was presented in the single dominant-eigenvalue regime primarily for clarity. We now state the natural extension to the case where extrapolation instability is governed by an $r$-dimensional leading eigenspace.

\begin{proposition}[Low-rank probabilistic amplification model]
\label{prop:low_rank}
Let $G \succeq 0$ be the Gram matrix from Definition~\ref{def:spectral_kappa} that admits spectral decomposition of the form,
\[
G = U \Lambda U^\top,
\qquad
\Lambda = \operatorname{diag}(\lambda_1,\dots,\lambda_d),
\]
with eigenvalues ordered as $\lambda_1 \ge \cdots \ge \lambda_d \ge 0$.
Assume that the leading $r$ eigenvalues dominate in the sense that,
\[
\lambda_1 \ge \cdots \ge \lambda_r \gg \lambda_{r+1} \ge \cdots \ge \lambda_d.
\]
Let $s \sim \mathrm{Unif}(S^{d-1})$ and recall the Rayleigh quotient
\[
\kappa^{\mathrm{prob}}_{\mathrm{spec}} := s^\top G s.
\]
Then,
\[
s^\top G s
= \sum_{i=1}^d \lambda_i Y_i^2,
\qquad
Y = U^\top s \sim \mathrm{Unif}(S^{d-1}),
\]
and the total energy in the dominant subspace is:
\[
\|P_{U_r} s\|_2^2
= \sum_{i=1}^r Y_i^2
\sim \mathrm{Beta}\!\left(\frac{r}{2}, \frac{d-r}{2}\right),
\]
where $U_r = \operatorname{span}\{u_1,\dots,u_r\}$. In particular, when $\lambda_{r+1},\dots,\lambda_d$ are negligible relative to
$\lambda_1,\dots,\lambda_r$, that is $\sum_{i=r+1}^d \lambda_i  Y_i^2 \ll \sum_{i=1}^r \lambda_i Y_i^2$ with high probability, we obtain the approximation:
\[
\kappa^{\mathrm{prob}}_{\mathrm{spec}}
\approx
\sum_{i=1}^r \lambda_i Y_i^2.
\]
\end{proposition}
\begin{proof}
By rotational invariance of the uniform distribution on $S^{d-1}$, we may assume that $U = I$ without loss of generality. Since a random variable which is uniformly distributed on a unit sphere is equivalent to a normalized vector of independent standard normal, we let $X \sim \mathcal{N}(0,I_d)$ and write,
\[
s = \frac{X}{\|X\|_2}.
\]
Combining the above leads to,
\[
\sum_{i=1}^r Y_i^2 = \sum_{i=1}^r s_i^2 = \frac{\sum_{i=1}^r X_i^2}{\sum_{i=1}^d X_i^2}.
\]
Since $\sum_{i=1}^r X_i^2 \sim \chi^2_r$ and
$\sum_{i=r+1}^d X_i^2 \sim \chi^2_{d-r}$ are independent,
the classical identity,
\[
\frac{U}{U+V}
\sim \mathrm{Beta}\!\left(\frac{r}{2},\frac{d-r}{2}\right) ,
\]
for independent $U \sim \chi^2_r$ and $V \sim \chi^2_{d-r}$
yields the main claim. The last conclusion is naturally derived as $\kappa^{\mathrm{prob}}_{\mathrm{spec}} = \sum_{i=1}^r \lambda_i Y_i^2 + \sum_{i=r+1}^d \lambda_i Y_i^2$.
\end{proof}

Proposition~\ref{prop:low_rank} shows that extrapolation instability is governed by the effective dimension of the unstable eigenspace rather than merely by the largest eigenvalue. 

In the rank-one case ($r=1$), extrapolation instability is concentrated along a single direction, and amplification depends on how strongly the coefficient error aligns with that eigenvector. In the general low-rank case, instability is governed by an $r$-dimensional unstable subspace. The random amplification factor is determined by the squared projection of the normalized error direction onto this subspace. The Beta distribution above, therefore, quantifies how much of a random error vector typically falls into the unstable region of coefficient space.

To summarize, when the leading spectrum is effectively $r$-dimensional, $\kappa^\rho_{\mathrm{spec}}$ is governed by the quantiles of the projected energy as $\mathrm{Beta}(r/2,(d-r)/2)$, scaled by the leading eigenvalues $\lambda_1,\ldots,\lambda_r$ (it scales like pure Beta only when $\lambda_1=\lambda_2=\cdots=\lambda_r$). Consequently, we expect the probabilistic anchor radii to remain smaller than the worst-case spectral constant $\lambda_{\max}(G)$ whenever $r \ll d$, while the advantage diminishes as $r$ approaches $d$.


\subsection{Constructing anchor functions}

The extrapolation bounds derived above provide certified tolerance parameters for anchor functions in terms of the true in-domain error $E_\Omega$. In practice, $E_\Omega$ is replaced by its empirical estimate $\widetilde E_\Omega$. Under standard noise assumptions, $\widetilde E_\Omega$ serves as a reliable upper approximation, yielding a certified feasible ball of radius $(\kappa \widetilde E_\Omega)^\frac{1}{2}$.

To reduce the size of the minimal feasible set, we propose a sampling-based strategy that generates a set of informative anchor functions. These anchor functions serve as local approximations that restrict the feasible set while preserving the underlying function's coverage. The core idea is that by optimizing over randomly chosen subsets of basis functions, we can construct diverse anchor functions whose extrapolation regions collectively constrain the search domain. 

The process begins by randomly selecting a subset of basis functions from the full set $\{\phi_k\}_{k=1}^d$. For each subset, we optimize the associated coefficients to construct an anchor function that fits the observed data in $\Omega$; any algorithm can be used, such as LASSO (Least Absolute Shrinkage and Selection Operator), LS (Least Squares), or any preferred algorithm. We then use extrapolation bounds (Theorems~\ref{theorem:all_kappa_calculation} and~\ref{theorem:inner_kappa_calculation}) to determine the validity region of each anchor function, i.e., where the error remains within an acceptable range.

Sampling different basis subsets introduces diversity among anchor functions, which helps shrink the feasible set more effectively. Importantly, since the minimal feasible set is an intersection of all feasible sets, additional anchor functions can only decrease the volume of the minimal feasible set $\mathcal{S}$, while still ensuring that the true function $f$ lies within it, provided that the extrapolation constraint $\mathcal{E}_\Omega \leq \Tilde{\mathcal{E}}_\Omega$ is maintained.

This process is formalized in Algorithm~\ref{alg:creating_anchor_functions}.

\begin{algorithm}[hb]
\caption{Creating Anchor Functions}\label{alg:creating_anchor_functions}
\begin{algorithmic}
\STATE {\bfseries Input:} Domain $\Omega$, target region $\Xi$, sample points $\{x_i\}_{i=1}^N \subset \Omega$, and basis functions $\{\phi_k\}_{k=1}^d$ of the function space $\mathcal{F}$.
\STATE {\bfseries Hyperparameters:} Degree of anchor function $m$, number of anchor functions $M$.
\STATE Initialize: Anchor function list $\mathcal{A} \gets \emptyset$.

\WHILE{$|\mathcal{A}| < M$}
    \STATE Randomly sample $m$ basis functions $\{\phi_{k_j}\}_{j=1}^m$.
    \STATE Find the coefficients $\{c_j\}_{j=1}^m$ using an optimization procedure (e.g., LS fitting on $\{x_i\}$)
    \STATE Define the anchor function $a(x) = \sum_{j=1}^m c_j \phi_{k_j}(x)$.
    \STATE Use whichever certificate is computable stably, together with empirical estimation $\Tilde{\mathcal{E}}_\Omega$: $$\delta \gets \sqrt{\kappa_{\mathrm{spec}}}\, \Tilde{\mathcal{E}}_\Omega \quad \text{or} \quad \delta \gets \sqrt{\kappa_{r}}\, \Tilde{\mathcal{E}}_\Omega .$$
    \IF{$a$ is not already in $\mathcal{A}$}
        \STATE Add $a$ to $\mathcal{A}$.
    \ENDIF
\ENDWHILE

\STATE \textbf{Return} $\mathcal{A}$

\end{algorithmic}
\end{algorithm}


\begin{remark}[Functional trust-region interpretation of Algorithm~\ref{alg:creating_anchor_functions}]
The construction of anchor functions in Algorithm~\ref{alg:creating_anchor_functions} can be interpreted as a functional analogue of trust-region optimization. While classical trust-region methods restrict parameter updates to regions where a local model is reliable~\cite{NocedalWright2006, ConnGouldToint2000}, each anchor $a_i$ with tolerance $\delta_i$ here defines a functional trust region that constrains admissible behavior on the extrapolation domain~$\Xi$. This viewpoint clarifies how anchor feasibility enforces stability on the extrapolation domain without altering the underlying approximation procedure.
\end{remark}

\subsection{Asymptotic limit under exact certification}

We consider the idealized case in which $\mathcal{E}_\Omega(a)$ is known exactly for every anchor $a \in \mathcal{F}$. We analyze the intersection of all corresponding feasible constraints and show that the minimal feasible set reduces to $\{f\}$. This characterizes the asymptotic limit of the framework under exact certification. To formalize this idealized setting, we consider a sequence of anchor functions 
$\{a_n\}_{n=0}^\infty \subset \mathcal{F}$ that is dense in $\mathcal{F}$. 

In the next theorem, for simplicity, we use the spectral bound of Theorem~\ref{theorem:all_kappa_calculation_improved}, though any valid certification constant would suffice.
\begin{theorem}
Let \( \{\phi_k\}_{k=1}^d \) be an orthogonal basis over the domain \( \Omega \), and let the associated function space be \( \mathcal{F} = \spann{\phi_k}_{k=1}^d \). Assume that the target function \( f \in \mathcal{F} \), and that for any \( a \in \mathcal{F} \), the extrapolation error \( \mathcal{E}_\Omega(a) \) is known exactly.

Then, the minimal feasible set \( \mathcal{S} \), defined as the intersection of all extrapolation constraints associated with the anchor functions, for all anchor functions as described in this section, converges to the singleton set:
\[
\mathcal{S} = \{f\}.
\]
\label{theorem:ideal_feasible_space}
\end{theorem}

\begin{proof}
Without loss of generality, assume \( f \equiv 0 \); otherwise work in the shifted space
\[
\mathcal{F} - f = \{h - f : h \in \mathcal{F}\}.
\]

For any anchor \( a \in \mathcal{F} \), using Theorem~\ref{theorem:all_kappa_calculation_improved}
\[
\mathcal{E}_\Xi(a) \le \sqrt{\kappa_{\mathrm{spec}}} \mathcal{E}_\Omega(a) = \sqrt{\kappa_{\mathrm{spec}}} \|a\|_\Omega.
\]
By Lemma~\ref{lemma:search_space_size_upper_bound}, the feasible set \( \mathcal{S}_i \) associated with anchor \( a_i \) satisfies
\[
\forall h \in \mathcal{S}_i:\quad \mathcal{E}_\Xi(h) \le 2\sqrt{\kappa_{\mathrm{spec}}} \|a_i\|_\Omega.
\]
Define
\[
\mathcal{S} := \bigcap_i \mathcal{S}_i.
\]
Following Theorem~\ref{theorem:all_kappa_calculation_improved}, \( f \in \mathcal{S}_i \) for all \( i \), hence \( f \in \mathcal{S} \).

For each \( n \in \mathbb{N} \), let
\[
A_n := \{ i : \|a_i\|_\Omega < 1/n \}, \qquad
\mathcal{S}_n := \bigcap_{i \in A_n} \mathcal{S}_i.
\]
Each \( \mathcal{S}_n \) is nonempty (since \( f \in \mathcal{S}_n \) for all \( n \)) and closed (intersection of closed sets). Moreover, \( A_{n+1} \subseteq A_n \), so \( S_{n+1} \subseteq \mathcal{S}_n \); thus \( (\mathcal{S}_n)_{n\in\mathbb{N}} \) is nested. 

For any \( h \in \mathcal{S}_n \) and any \( i \in A_n \),
\[
\mathcal{E}_\Xi(h) \le 2\sqrt{\kappa_{\mathrm{spec}}} \|a_i\|_\Omega \le \frac{2\sqrt{\kappa_{\mathrm{spec}}}}{n}.
\]
Since \( f \equiv 0 \), we may write \( \mathcal{E}_\Xi(h) = \|h\|_\Xi \), hence
\[
\|h\|_\Xi \le \frac{2\sqrt{\kappa_{\mathrm{spec}}}}{n}.
\]
Thus for any \( h_1, h_2 \in \mathcal{S}_n \),
\[
\|h_1 - h_2\|_\Xi \le \|h_1\|_\Xi + \|h_2\|_\Xi \le 4\frac{\sqrt{\kappa_{\mathrm{spec}}}}{n},
\]
so \( \mathrm{diam}_\Xi(\mathcal{S}_n) \to 0 \) as \( n \to \infty \).

The space \( \mathcal{F} \) is finite-dimensional, hence complete under \( \|\cdot\|_\Xi \). By Cantor's intersection theorem in complete metric spaces
(e.g., \citep[see, e.g., Chapter~2]{folland1999real}), the intersection
\[
T := \bigcap_{n=1}^\infty \mathcal{S}_n
\]
consists of a single function; denote it by \( h^\star \).

Since each \( \mathcal{S}_n \) is an intersection over a subset of the indices used to define \( \mathcal{S} \), we have
\[
\mathcal{S} = \bigcap_i \mathcal{S}_i \subseteq \bigcap_{n=1}^\infty \mathcal{S}_n = \{h^\star\},
\]
so \( \mathcal{S} \subseteq \{h^\star\} \). On the other hand, \( f \in \mathcal{S} \), hence \( h^\star = f \) and therefore \( \mathcal{S} = \{f\} \).
\end{proof}

We conclude this part with three remarks. 

First, we model our asymptotic setting by considering a sequence of anchor functions 
$\{a_n\}_{n=0}^\infty \subset \mathcal{F}$ that is dense in $\mathcal{F}$. This scenario can be materialized, for example, by starting from the least-squares solution $a_0$ that minimizes $\|f-a\|_\Omega^2$ and then perturbing its coefficients along basis directions with amplitudes tending to zero. This generates anchors that explore $\mathcal{F}$ with arbitrary fineness while maintaining controlled in-domain error. 

Second, Theorem~\ref{theorem:ideal_feasible_space} demonstrates that in this optimistic setting, where the errors \( \mathcal{E}_\Omega(a) \) are known exactly for all \( \mathcal{S} \), the feasible set collapses to the true function \( f \). In practice, however, such exact knowledge is unattainable due to measurement noise and model uncertainty, so this result should be understood as a theoretical bound rather than a practical outcome.

Third, if the sampling density on $\Omega$ increases so that the empirical error $\widetilde E_\Omega$ converges to the true error $E_\Omega$, then the feasible region constructed from the certified bounds converges to that of
Theorem~\ref{theorem:ideal_feasible_space}. In this limit, the feasible set reduces to $\{f\}$, implying consistency of the projection-based estimator under exact certification.

\section{Numerical Experiments and Applications} \label{sec:ExperimentalResults}

This section provides a systematic validation of the theoretical results developed in Section~\ref{sec:extrapolation method}. We organize our numerical examples according to the analytical components they are designed to verify. First, we examine the projection guarantee of Theorem~\ref{thm:gurantee1} in progressively richer settings, beginning with a minimal synthetic example that isolates the geometric mechanism of feasibility and projection, and culminating in a real-world geomagnetic dataset. We then study the sharpness and numerical behavior of the extrapolation bounds derived in Section~\ref{section:creating_anchor_functions}, including comparisons between worst-case and spectral constants. Finally, we demonstrate the practical impact of the probabilistic certification framework on manifold and PDE-based problems. Throughout, we report extrapolation performance using the root error functional $\mathcal{E}_D(\cdot)$ (see Definition~\ref{eq_ie_mse_objective_2}), to preserve the physical units of the underlying problem.

In Subsections\ref{subsec:numerical_projection}-\ref{subsec:certification}, we state several theoretical bounds for bases that are orthogonal (or orthonormal) with respect to the in-domain inner product on $\Omega$. In our experiments, $\Omega=[-1,c]$ is a strict subinterval of the canonical Legendre orthogonality domain $[-1,1]$, so the raw Legendre basis is not orthogonal on $\Omega$. Accordingly, we compute all bound-related quantities (in particular the extrapolation constant $\kappa_{spec}$) using a basis obtained by reparameterizing the Legendre functions to the $\Omega$-domain via an affine change of variables. Concretely, letting $c\in(-1,1)$ be the cutoff, we map $u\in[-1,1]$ to $t\in[-1,c]$ by
\[
t = \alpha u + \beta, \qquad \alpha=\frac{c+1}{2},\ \ \beta=\frac{c-1}{2},
\]
and define the $\Omega$-aligned basis
\[
\tilde\varphi_k(t) \;:=\; \varphi_k\!\left(\frac{t-\beta}{\alpha}\right), \qquad k=0,\dots,d,
\]
which is orthogonal on $t\in[-1,c]$ with respect to the pulled-back $L^2$ inner product. We then compute $\kappa_{spec}$ using $\{\tilde\varphi_k\}$ on $\Omega$ and $\Xi=[c,1]$.
All fitted predictors (LS/Ridge/LASSO) and all reported errors are evaluated directly in the original $t$-coordinates using the standard Legendre coefficient representation; i.e., the basis reparameterization is used only to enforce the orthogonality assumptions for the theoretical constants, not to alter the learned functions.

\subsection{Projection Guarantees in Practice} \label{subsec:numerical_projection}

We begin with a minimal setting designed to isolate the geometric effect of projection onto a certified feasible set. The goal is to examine how the feasibility constraint alters extrapolation behavior independently of model complexity or application-specific structure.

\subsubsection{Extrapolating with simple anchors}
\label{sec:easily_approximated_anchors}

We consider the damped oscillator
\[
f(t) = e^{-\zeta \omega_0 t}\Big(\cos(\omega_1 t) + \tfrac{\sqrt{\zeta}}{\sqrt{1-\zeta^2}} \sin(\omega_1 t)\Big),
\quad \omega_1=\omega_0\sqrt{1-\zeta^2},
\]
with $\zeta=0.3$ and $\omega_0=2\pi$.
The main motivation of this experiment is to show that anchors do not need to be complicated. In many real-world scenarios, practitioners can easily provide bounds or envelopes for the function, even when its exact form is unknown. Such bounds, for example, constant upper/lower limits or monotonic decay ranges, are straightforward to approximate and thus serve as natural anchors. To illustrate this, we use the simplest possible case: a constant anchor $a(t)\equiv 0$, together with symmetric bounds of $\pm 1.2$ around it. These bounds are encoded by a worst-case feasible radius $\delta_{\max} = 1.2$ in terms of the $\mathcal{E}$ error, corresponding to the error of a constant function that saturates the upper bound. In practice, one can select a more restrictive bound and obtain greater improvements.

The domain $T=[-1,1]$ is split at $t=0$ into $\Omega=[-1,0]$ (fit region) and $\Xi=(0,1]$ (extrapolation region). The oscillator is sampled on grids of size $|\Omega|=100$ and $|\Xi|=400$, with additive Gaussian noise of standard deviation $\sigma=0.5$ applied only in $\Omega$. To estimate $\mathcal{E}_\Xi$, we approximate the error integral from samples using Simpson's rule. We fit a degree–12 Legendre model using LASSO on $\Omega$, with $\alpha=0.01$, and then project it onto the anchor’s feasible ball in $\Xi$. As guaranteed by our projection theory, this operation cannot increase the $\Xi$-error and in fact strictly improves extrapolation as the original solution lies outside the feasible region.

The extrapolation errors achieved, and demonstrated in Figure~\ref{fig:oscillator_projection}:
\[
\mathcal{E}^{LASSO}_\Xi=1.737,  \quad \mathcal{E}_\Xi^{Projection}=1.153 \, .
\]
Constituting an error reduction of 0.5846. Throughout this section, we denote by $\Delta \mathcal{E}_\Xi$ this reduction in extrapolation error, as done by projection, over the domain $\Xi$.

\begin{figure}[tb]
  \centering
  \includegraphics[width=0.8\linewidth]{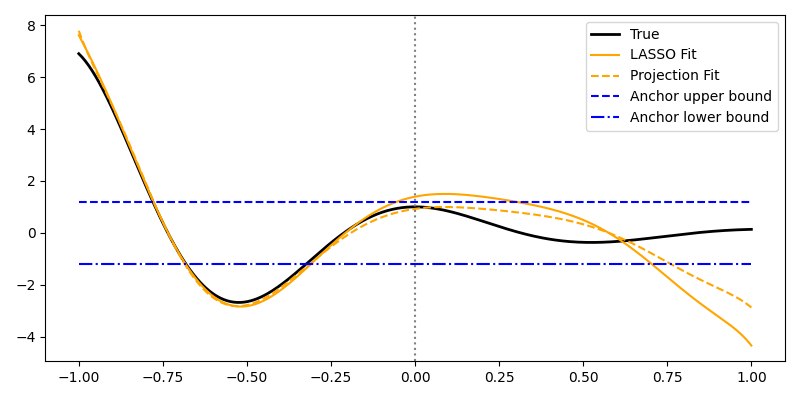}
  \caption{True oscillator (black), LASSO fit (solid orange), projected fit (dashed orange), and anchor bounds (blue, $\pm 1.2$ around 0). The vertical line marks the split between $\Omega$ and $\Xi$. Projection toward the constant anchor substantially reduces extrapolation error.}
  \label{fig:oscillator_projection}
\end{figure}

This experiment shows that simple bounds, which are often readily available in applications, can be used as anchor functions: even a crude constant anchor with rough bounds provides a useful constraint that, when combined with projection, guarantees improvement (or at least no deterioration) in extrapolation error. Furthermore, Theorem~\ref{theorem:5} also provides explicit bounds on the improvement in extrapolation error achieved by projection:
\[
0.3540
\leq
\Delta\mathcal{E}_\Xi
= 0.5846
\leq
0.6098 \,.
\]
Thus, the empirically observed reduction in extrapolation error lies neatly between
the theoretically predicted lower and upper bounds, providing another concrete
verification of the theory.

\begin{remark} [Why does the projection still violate the bounds]
In Figure~\ref{fig:oscillator_projection}, the projected solution can still violate the original pointwise bounds. The reason is that we enforce the anchor constraint in an \(L_2(\Xi)\)-ball around \(a\equiv 0\), whereas the prescribed bounds are pointwise (\(L_\infty\)) constraints; converting the latter into an \(L_2\) radius via a conservative worst-case mapping yields a large feasible set. Consequently, the feasible set may contain functions that violate the pointwise bounds, and the projection may land on such a function. In practice, one may use a tighter radius to ensure the projection respects the intended bounds.
\end{remark}

\subsubsection{Geometric Interpretation of the Projection Step}

We next examine the geometric structure underlying the projection guarantee of Theorem~\ref{thm:gurantee1}. The goal of this experiment is not merely to demonstrate improvement, but to illustrate how the distance to the feasible set and the orientation of the error jointly determine the magnitude of the extrapolation reduction. By visualizing both the feasible region and the relative positions of competing predictors, we obtain a direct geometric interpretation of the lower and upper improvement bounds.

To examine the geometric behavior of the projection step, we consider a high-degree polynomial setting in which extrapolation is known to be unstable. In this experiment, we consider the interval $[-1,1]$ and define the ground truth as the sum of the first $20$ Legendre basis functions. We set $\Omega = [-1,0.9)$ and draw $50$ training samples from this region. The observations are perturbed by additive Gaussian noise with standard deviation $\sigma = 0.01$. Figure~\ref{fig:theorem_5_experiment_ls_lasso_dist_visualization} compares the extrapolation performance of least squares (LS), LASSO (with $\alpha = 0.001$), and the projected solution obtained via Theorem~\ref{theorem:5} over the extrapolation domain $\Xi = [0.9,1]$. In this configuration, the spectral extrapolation constant computed according to Theorem~\ref{theorem:all_kappa_calculation_improved} is $\kappa_{\mathrm{spec}} = 372{,}167$, reflecting the severe amplification induced by the high polynomial degree. Visualized in Figure~\ref{fig:theorem_5_experiment_ls_lasso_dist_visualization}, the LS method strongly overfits the noise. 

We construct the feasible set with radius 
$\sqrt{\kappa_{\mathrm{spec}}}\,E_\Omega(g_{\mathrm{LS}})$ around the LS solution. This set is certified by Theorem~\ref{theorem:all_kappa_calculation_improved} to contain the true function on~$\Xi$. Since the ball is centered at the LS solution, it trivially contains the LS predictor itself. The LASSO prediction, by contrast, lies outside this feasible set.

\begin{figure}[htb]
    \centering
    \begin{subfigure}[]{0.35\textwidth}
    \includegraphics[width=\textwidth]{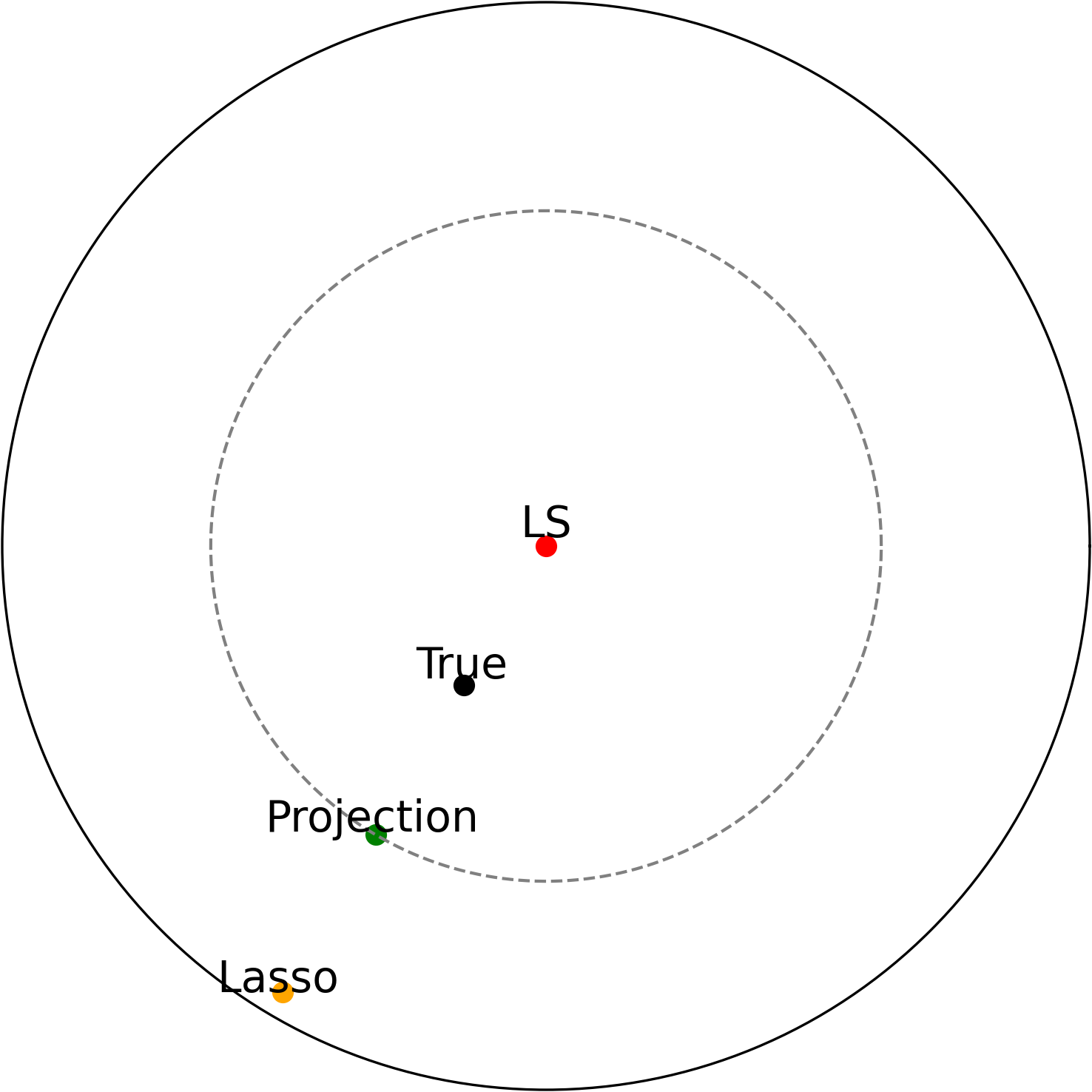}
        \caption*{}    
    \end{subfigure} \qquad
    \begin{subfigure}[]{0.42\textwidth}
        \includegraphics[width=\textwidth]{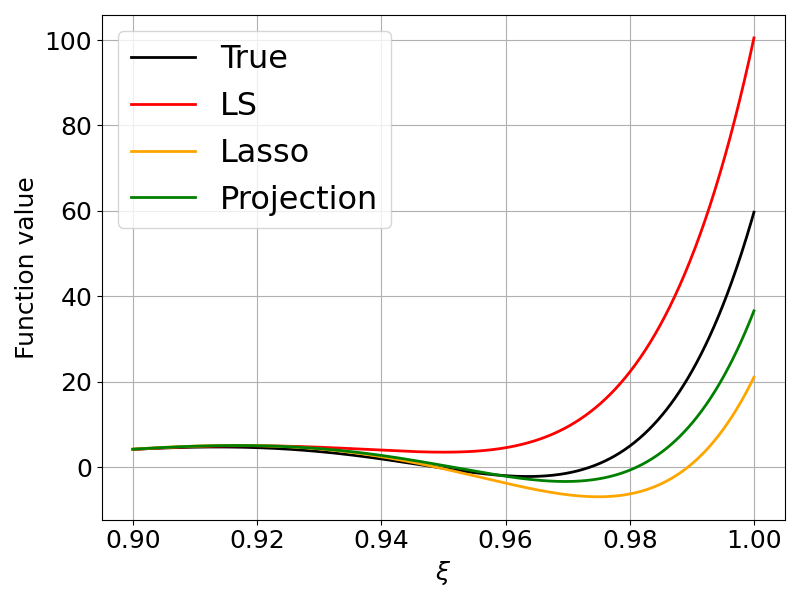}
        \caption*{}  
    \end{subfigure}
\caption{Extrapolation behavior over $\Xi$ for LS, LASSO ($\alpha=0.001$), and the projected solution of Theorem~\ref{theorem:5}. 
Left: distance–orientation plot relative to the LS approximation; the dashed circle denotes the feasible-set boundary. 
Right: true function (black), LS (red), LASSO (orange), and projection (green). 
Reported root errors on $\Xi$ are LS: 4.39, LASSO: 3.51, Projection: 2.00, with improvement bounds $0.30 \le \Delta E_\Xi \le 1.54$.}
\label{fig:theorem_5_experiment_ls_lasso_dist_visualization}
\end{figure}

Consequently, as stated in Theorem~\ref{theorem:5}, projecting LASSO’s solution onto the feasible set yields a function with strictly better extrapolation performance over $\Xi$. Quantitatively, the $\Xi$-domain errors with respect to the true function are
\[
\mathcal{E}_\Xi^{LS}: 4.39,\qquad
\mathcal{E}_\Xi^{LASSO}: 3.51,\qquad
\mathcal{E}_\Xi^{Projection}: 2.00,
\]
demonstrating that the projected solution substantially improves upon both LS and LASSO in this experiment. Importantly, Theorem~\ref{theorem:5} also guarantees upper and lower error improvement. In our setting, the reduction in the extrapolation-error bound achieved by the projected solution satisfies
\[
0.30 \leq \Delta\mathcal{E}_{\Xi}=1.51 \leq 1.54,
\]
The observed reduction in $\Xi$-error from LASSO to the projection is $1.51$, which is very close to the upper bound $1.54$. This near-saturation of the bound can be attributed to the strong alignment between the orientations of the true function and the LASSO solution over $\Xi$.

\subsubsection{Extrapolation on a real-world geomagnetic dataset}

Having examined the projection mechanism in controlled synthetic settings, we now assess its behavior on real observational data. This experiment tests whether the non-worsening guarantee and quantitative improvement bounds of Theorem~\ref{theorem:5} persist in a realistic scientific application where the target function is not analytically prescribed and the anchor radii must be estimated from empirical in-domain error.

For our real-world application, we utilize a geomagnetic dataset derived from a one-dimensional meridional cross-section of the Earth's main magnetic field. This field is represented by the radial component $B_r$ as a function of $\mu = \cos\theta \in [-1,1]$. The choice of $\mu$ places the problem naturally in a Legendre polynomial setting, making polynomial approximation in a Legendre basis particularly appropriate. We partition the domain $T=[-1,1]$ at $\mu=0.8$ into a fitting region $\Omega=[-1,0.8]$ and an extrapolation region $\Xi=[0.8,1]$, thereby enforcing a nontrivial out-of-domain prediction task near the polar cap. On $\Omega$, we fit two degree-8 polynomial models using LS and ridge regression (the latter with an $\ell^2$-regularization parameter $\alpha=10^{-1}$). For this configuration, the corresponding extrapolation condition number is $\kappa_{\mathrm{spec}}=309$, indicating moderate but non-negligible amplification of in-domain error under extrapolation.

The field $B_r(\mu)$ is computed from the official WMM--2025 Gauss coefficients by evaluating the standard spherical harmonic expansion (associated Legendre polynomials) at Earth’s reference radius and a fixed longitude, then sampling along one meridian. The dataset consists of $N=721$ uniformly spaced points over $\mu\in[-1,1]$ (spacing $\Delta\mu\approx 2.78\times 10^{-3}$). Since the underlying continuous field is not available in closed form, we construct a continuous interpolated signal, which serves as a surrogate, $f(\mu)$ from these samples using a shape-preserving cubic Hermite interpolant (PCHIP) \cite{FritschButland1984, FritschCarlson1980} on $[-1,1]$ (with linear extrapolation outside the sampled range). This allows us to evaluate $f$ on an arbitrary numerical grid, without being constrained by the number of measured points.

For numerical integration and reporting, we evaluate models on a uniform grid $T=\{\mu_j\}_{j=1}^{400}\subset[-1,1]$. The fit region $\Omega$ is the subset of grid points with $\mu_j\le 0.8$ (yielding $|\Omega|=360$ points). For evaluating extrapolation performance, we use a dense uniform grid of $400$ points over $\Xi=[0.8,1]$ to reduce discretization error in the approximation of $\mathcal{E}_\Xi$. This grounds the experiment in real observational data while keeping the reported errors stable with respect to the evaluation grid. The Gauss coefficients were obtained from the official NOAA/NCEI World Magnetic Model release \cite{noaa_wmm_coefficients}.

To demonstrate our ability to improve extrapolation while creating anchor functions directly from fitted models (without prior knowledge), we compute feasible radii using Theorem~\ref{theorem:all_kappa_calculation_improved}. In practice the in-domain error is estimated from noisy observations on $\Omega$; accordingly, the feasible radius is formed using an empirical proxy $\tilde{\mathcal{E}}_\Omega$ for $\mathcal{E}_\Omega$. This experiment therefore constitutes a direct numerical verification of the single-anchor projection guarantee in Theorem~\ref{theorem:5}. 

\begin{table}[tb]
\centering
\caption{Errors on $\Omega$ (fit) and $\Xi$ (extrapolation), $\kappa_{\mathrm{spec}}=309$, cutoff $=0.8$.
The final column reports the extrapolation error on $\Xi$ after projection of
each model into the other’s feasible ball; the smallest projected error is in bold.}
\label{tab:geomag_results}
\begin{tabular}{lcccc}
\toprule
Method & $\mathcal{E}_\Omega$ & $\mathcal{E}_\Xi$ (vs.\ true) & Projected $\mathcal{E}_\Xi$ (vs.\ true) \\
\midrule
LS    & $7.542\times 10^{2}$ & $1.297\times 10^{4}$ & $1.067\times 10^{4}$ \\
Ridge & $9.466\times 10^{2}$ & $6.174\times 10^{3}$ & $\mathbf{1.282\times 10^{3}}$ \\
\bottomrule
\end{tabular}
\end{table}

As summarized in Table~\ref{tab:geomag_results}, projection reduces the extrapolation root error $\mathcal{E}_\Xi$ of the LS fit from $1.297\times 10^{4}$ to $1.067\times 10^{4}$, and that of the ridge regression fit from $6.174\times 10^{3}$ to $1.282\times 10^{3}$, with the projected ridge regression model achieving the best performance and yielding a substantial improvement over the raw fits. Since $\kappa$ is invariant across fits and the feasible radius is $\delta=\sqrt{\kappa_{\mathrm{spec}}}\,\mathcal{E}_\Omega$, the model with smaller error on $\Omega$ induces the smaller feasible set. Because LS is unregularized, it attains the lower $\mathcal{E}_\Omega$ and thus defines the smaller feasible ball; in practice, one would therefore project the ridge regression model onto the LS feasible set, rather than the other way around. We perform both projections here for completeness.

\begin{figure}[tb]
  \centering
  \begin{subfigure}[b]{0.49\linewidth}
    \centering
    \includegraphics[width=\linewidth]{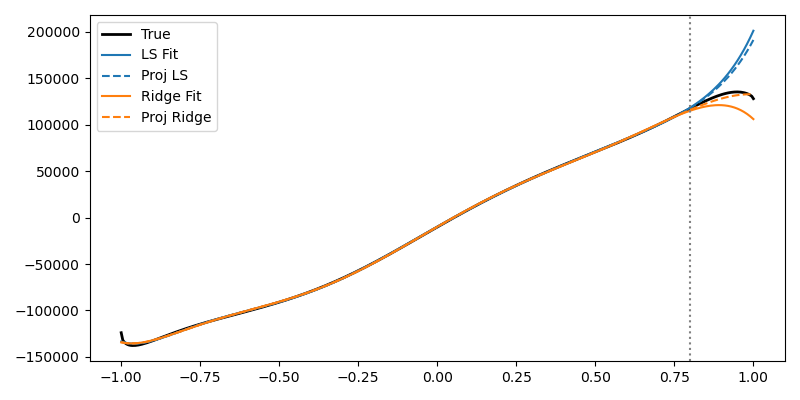}
    \caption{Full view on $\Omega \cup \Xi$.}
    \label{fig:geo_ls_lasso_full}
  \end{subfigure}
  \hfill
  \begin{subfigure}[b]{0.49\linewidth}
    \centering
    \includegraphics[width=\linewidth]{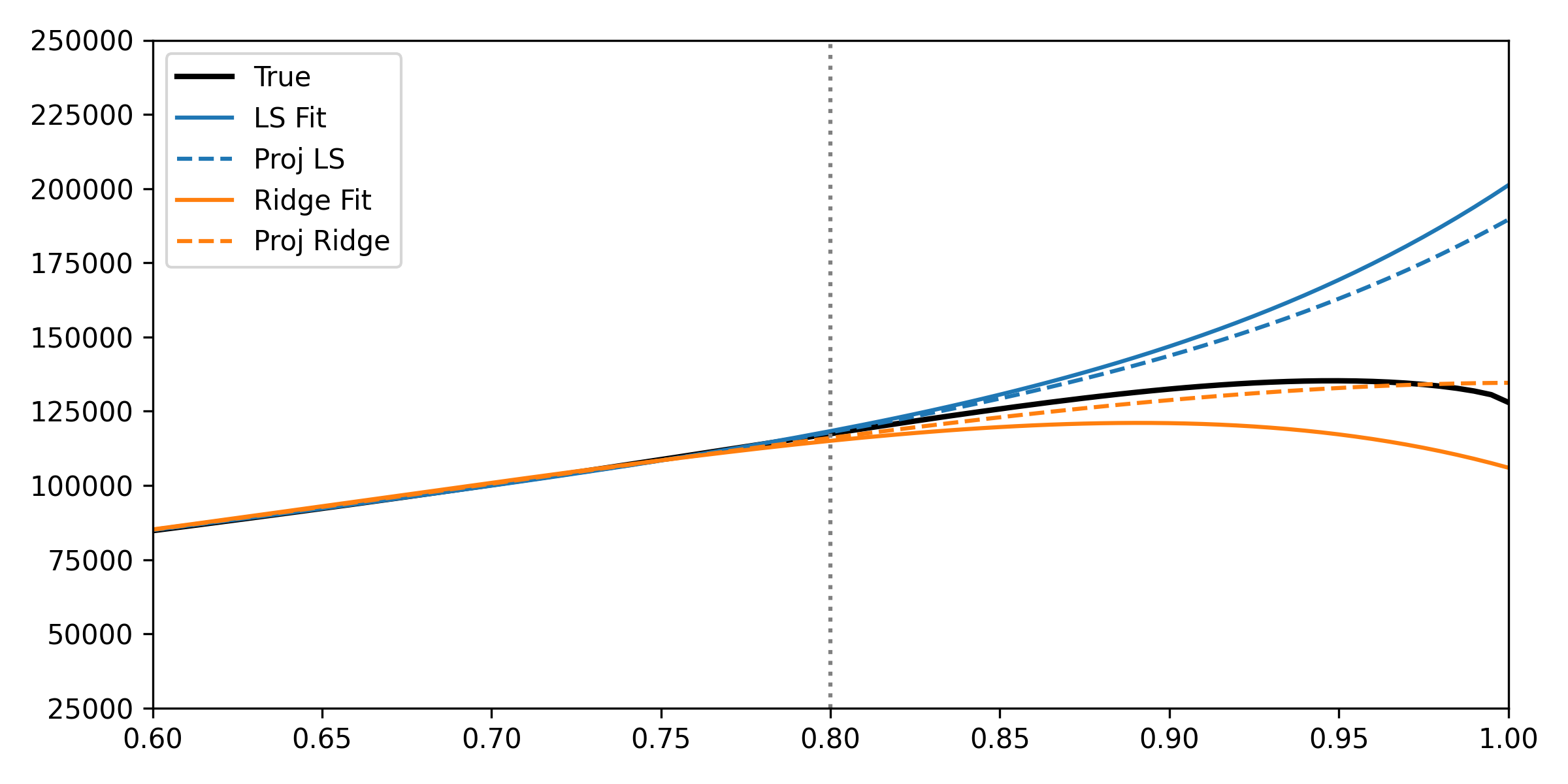}
    \caption{Zoomed-in view on the extrapolation region $\Xi$.}
    \label{fig:geo_ls_lasso_zoom}
  \end{subfigure}
  \caption{LS and ridge regression fits on $\Omega$ and their projected counterparts on $\Xi$
  for $B_r(\mu)$ (degree $8$, $\sigma=0.05$, $\kappa_{\mathrm{spec}}=309$). Projection substantially
  reduces extrapolation error on $\Xi$ for this real-world geomagnetic dataset.}
  \label{fig:geo_ls_lasso_projection}
\end{figure}

The corresponding fits and their projected counterparts are shown in Figure~\ref{fig:geo_ls_lasso_projection}. While both LS and ridge regression provide accurate interpolation over $\Omega$, substantial amplification is visible near the boundary of $\Xi$. The projection step visibly suppresses this boundary overshoot, yielding a smoother and more stable extrapolation profile without altering the in-domain fit. This qualitative behavior is consistent with the non-worsening guarantee and quantitative improvement bounds established in Theorem~\ref{theorem:5}.

Theorem~\ref{theorem:5} also yields explicit lower and upper bounds on the extrapolation improvement. For the LS fit we obtain
\[
1.267\times 10^{3}
\;\le\;
\Delta\mathcal{E}_\Xi^{\mathrm{LS}}
= 2.300\times 10^{3}
\;\le\;
2.312\times 10^{3},
\]
and for ridge regression,
\[
3.123\times 10^{3}
\;\le\;
\Delta\mathcal{E}_\Xi^{\mathrm{Ridge}}
= 4.892\times 10^{3}
\;\le\;
5.697\times 10^{3}.
\]
In both cases, the observed reduction lies strictly within the theoretical interval, confirming the sharpness of the projection bounds in a real-data setting.

\subsection{Certification Quality and Spectral Bounds} \label{subsec:certification}

We now examine the numerical behavior and sharpness of the extrapolation bounds developed in Section~\ref{section:creating_anchor_functions}. While Theorem~\ref{theorem:all_kappa_calculation_improved} establishes a tight spectral constant and Theorem~\ref{theorem:inner_kappa_calculation} provides a numerically stable alternative, their practical relevance depends on how these constants compare with the classical worst-case bound of Theorem~\ref{theorem:all_kappa_calculation}. In this section, we quantify these differences empirically and investigate the extent to which the spectral and probabilistic refinements reduce conservatism in realistic approximation settings.

\subsubsection{Comparison of extrapolation upper bounds}

We begin by empirically validating Theorem~\ref{theorem:kappa_improved_always_better}, which asserts that the spectral extrapolation constant $\kappa_{\mathrm{spec}}$ of Theorem~\ref{theorem:all_kappa_calculation_improved} is always no larger than the classical condition number $\kappa$ of Theorem~\ref{theorem:all_kappa_calculation}. 

\begin{figure}[ht]
    \centering
    \begin{subfigure}[]{0.32\textwidth}
    \includegraphics[width=\textwidth]{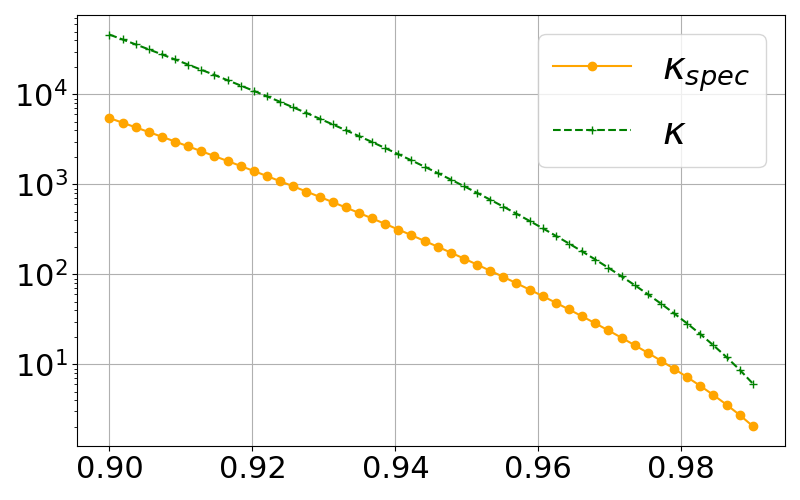}
        \caption*{}    
    \end{subfigure}
    \begin{subfigure}[]{0.32\textwidth}
        \includegraphics[width=\textwidth]{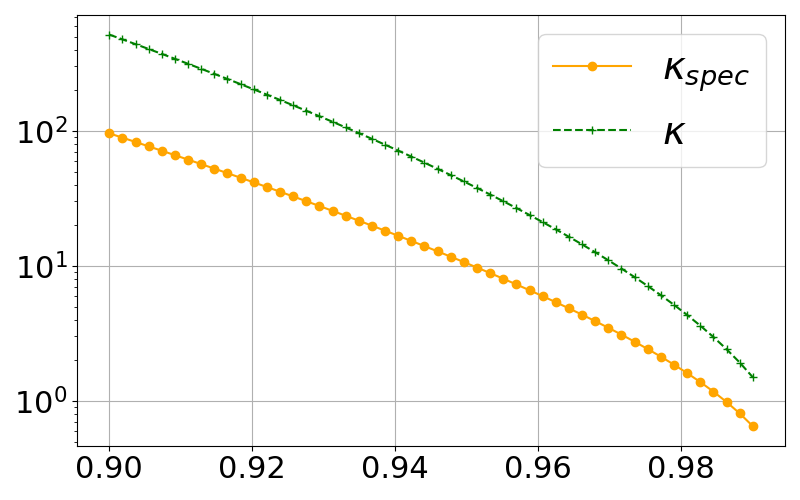}
        \caption*{}  
    \end{subfigure}
    \begin{subfigure}[]{0.32\textwidth}
        \includegraphics[width=\textwidth]{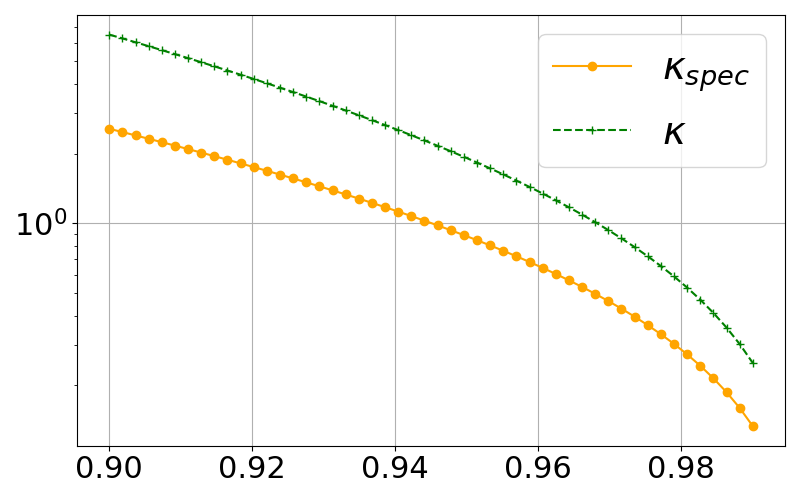}
        \caption*{}  
    \end{subfigure}
\caption{Comparison between upper bounds calculated using Theorem~\ref{theorem:all_kappa_calculation} ($\kappa$) and Theorem~\ref{theorem:all_kappa_calculation_improved} ($\kappa_{\mathrm{spec}}$). The plots represent $15$ basis functions (right), $10$ basis functions (middle), and $5$ basis functions (left) with a log-scaled y-axis.}
\label{fig:bound_comparison}
\end{figure}

Figure~\ref{fig:bound_comparison} compares these bounds for Legendre polynomial bases of sizes 5, 10, and 15, under varying cutoffs between $\Omega$ and $\Xi$. In all cases, the spectral bound is strictly tighter, confirming the theoretical dominance of $\kappa_{\mathrm{spec}}$ and illustrating the degree to which the classical worst-case constant can overestimate extrapolation amplification.

\begin{figure}[ht]
    \centering
    \begin{subfigure}[]{0.32\textwidth}
    \includegraphics[width=\textwidth]{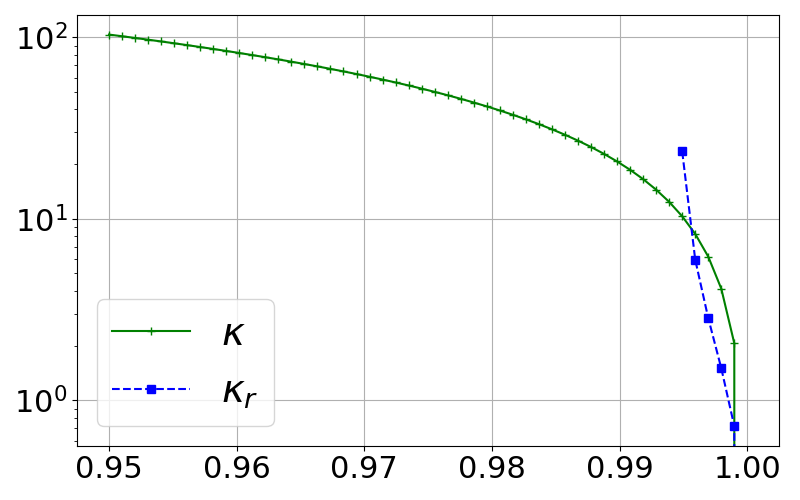}
        \caption*{}    
    \end{subfigure}
    \begin{subfigure}[]{0.32\textwidth}
        \includegraphics[width=\textwidth]{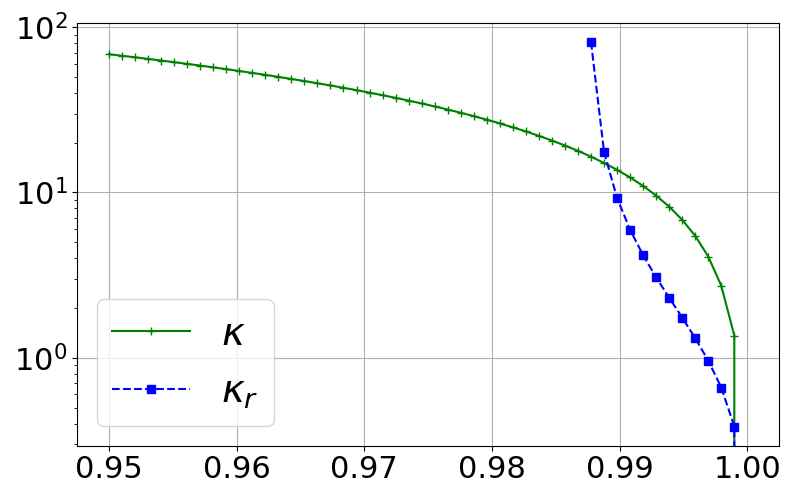}
        \caption*{}  
    \end{subfigure}
    \begin{subfigure}[]{0.32\textwidth}
        \includegraphics[width=\textwidth]{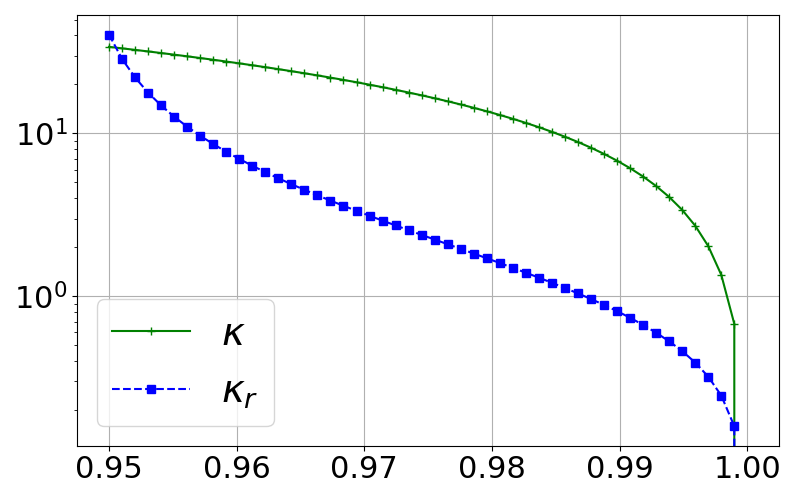}
        \caption*{}  
    \end{subfigure}
\caption{Comparison between upper bounds calculated using Theorems \ref{theorem:all_kappa_calculation} ($\kappa$) and \ref{theorem:inner_kappa_calculation}.
The latter corresponds to the inner-domain bound of Theorem~\ref{theorem:inner_kappa_calculation} and applies when condition~\eqref{eqn:cond_new_kappa} is satisfied. The plots show 15 basis functions (right), 10 basis functions (middle), and 5 basis functions (left) of the Chebyshev polynomials, with a log-scaled y-axis.}
\label{fig:bound_comparison_with_inner}
\end{figure}

As noted in Remark~\ref{remark:bound_comparison_with_inner}, the bound in Theorem~\ref{theorem:all_kappa_calculation} is governed by the behavior of the basis functions outside their orthogonality region, where instability typically increases near domain boundaries. In contrast, Theorem~\ref{theorem:inner_kappa_calculation} applies when both $\Omega$ and $\Xi$ lie within the orthogonality range, a regime in which the basis functions exhibit more regular behavior. Although the inner-domain condition is not always satisfied, in edge-dominated polynomial settings it often yields bounds that are tighter by an order of magnitude, as illustrated in Figure~\ref{fig:bound_comparison_with_inner}. 

We compute both bounds for Legendre polynomial bases of sizes 5, 10, and 15 under varying cutoff locations, following Theorem~\ref{theorem:all_kappa_calculation_improved}. Because these estimates avoid eigenvalue evaluation, they provide a numerically stable alternative to the spectral bound.

\subsubsection{Empirical validation of the probabilistic extrapolation condition number}
\label{subsec:empirical_prob_kappa_sphere}

We now examine the distributional accuracy of the probabilistic extrapolation model introduced in Section~\ref{subsec:prob_anchor_sphere}. In the near-dominant eigenvalue regime, Lemma~\ref{lem:rank_one_beta_sphere} and Corollary~\ref{cor:almost_rank_one_beta_sphere} predict that the normalized Rayleigh quotient 
\begin{equation} \label{eqn:normalized_Rayleigh_q}
    q(c) = \frac{c^\top (G_{\Xi}) \, c}{\|c\|_2^2} \,,
\end{equation}
is well-approximated by a scaled $\mathrm{Beta}\!\left(\tfrac12,\tfrac{d-1}{2}\right)$ distribution. This experiment evaluates how closely the empirical distribution of $q(c)$ follows this theoretical prediction and quantifies the reduction in conservatism relative to the worst-case spectral constant.


\begin{figure}[t]
    \centering
    \includegraphics[width=0.75\linewidth]{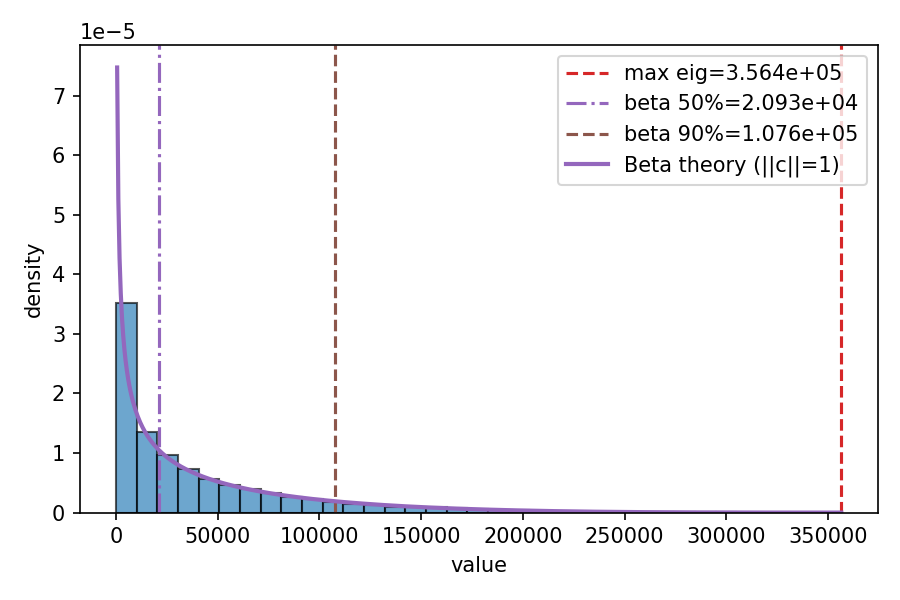}
    \caption{Empirical distribution of
    $q(c)=\frac{c^\top \tilde G_{\Xi}c}{\|c\|_2^2}$ (basis size $m=9$; $n=10{,}000$
    samples), compared to the Beta distribution predicted by Lemma~\ref{lem:rank_one_beta_sphere} and the near-dominant eigenvalue approximation of Corollary~\ref{cor:almost_rank_one_beta_sphere}. Vertical markers indicate the theoretical median ($50\%$), the theoretical $90\%$ quantile, and the worst-case $\lambda_{\max}(\tilde G_{\Xi})$.}
    \label{fig:probabilistic-eigs}
\end{figure}

We use a basis for quadratic spherical harmonics, which consists of $9$ elements in total, and sample $n=10{,}000$ coefficients to evaluate the empirical $\nicefrac{c^\top \tilde G_{\Xi}c}{\|c\|_2^2}$ that estimates~\eqref{eqn:normalized_Rayleigh_q}. This estimation is associated with $\kappa_{\mathrm{spec}}^{\rho}$ of
Definition~\ref{def:prob_spectral_kappa_sphere}, having $\tilde G_{\Xi}$ in place of $G$.

Figure~\ref{fig:probabilistic-eigs} displays the empirical histogram of $q(c)$ in~\eqref{eqn:normalized_Rayleigh_q}. 
In this experiment, the near-dominant eigenvalue condition of Corollary~\ref{cor:almost_rank_one_beta_sphere} is strongly satisfied:
\[
\lambda_{\max}(\tilde G_{\Xi}) = 3.5645\times 10^{5},\quad
\lambda_{2}(\tilde G_{\Xi}) = 2.5897\times 10^{3},\quad
\frac{\lambda_{\max}}{\lambda_2}\approx 1.38\times 10^{2}.
\]
The substantial spectral gap justifies the rank-one Beta approximation of Lemma~\ref{lem:rank_one_beta_sphere}. 
Consistent with this prediction, the empirical distribution closely matches the scaled Beta law, with quantiles
\[
\kappa^{0.5}_{\mathrm{spec}}=2.0927\times 10^{4},\qquad
\kappa^{0.9}_{\mathrm{spec}}=1.0757\times 10^{5},
\]
and corresponding empirical coverages
\[
\Pr[q(c)\le \kappa^{0.5}_{\mathrm{spec}}] \approx 0.5007,\qquad
\Pr[q(c)\le \kappa^{0.9}_{\mathrm{spec}}] \approx 0.9046.
\]

Finally, we observe how the probabilistic bounds substantially reduce the effective extrapolation condition number relative to the worst-case $\kappa_{\mathrm{spec}}=\lambda_{\max}(\tilde G_{\Xi})$.
In the typical ($50\%$) regime this yields
\[
\frac{\kappa_{\mathrm{spec}}}{\kappa^{0.5}_{spec}}
\approx \frac{3.5645\times 10^{5}}{2.0927\times 10^{4}}
\approx 17.0,
\]
and at $90\%$ confidence
\[
\frac{\kappa_{\mathrm{spec}}}{\kappa^{0.9}_{spec}}
\approx \frac{3.5645\times 10^{5}}{1.0757\times 10^{5}}
\approx 3.31.
\]
Thus, the probabilistic anchor bound \eqref{eq:prob_bound_root_sphere} provides substantially tighter typical- and high-confidence extrapolation guarantees than the worst-case spectral bounds of Theorem~\ref{theorem:all_kappa_calculation} and Theorem~\ref{theorem:inner_kappa_calculation}.

\subsection{Manifold and PDE Applications}

We now evaluate the projection framework in geometrically and physically structured settings. Unlike the previous polynomial experiments, the following examples involve multivariate functions defined on curved manifolds and fields generated by partial differential equations. These settings test whether the feasibility-and-projection mechanism, together with the proposed certification strategies, remains effective in realistic scientific computing contexts where geometry, discretization, and noise interact.

\subsubsection{Extrapolation on the sphere via spherical harmonics}
\label{subsec:spherical_experiment}

We next evaluate the projection framework in a genuinely geometric setting, where the underlying function space is induced by the eigenfunctions of the Laplace–Beltrami operator on the sphere $S^2 \subset \mathbb{R}^3$. This experiment tests whether the feasibility-and-projection mechanism remains effective when extrapolation occurs across curved domains and the basis functions are intrinsic manifold harmonics rather than one-dimensional polynomials.

Let $\{Y_\ell^m\}_{\ell \ge 0,\,-\ell \le m \le \ell}$ denote the real spherical harmonics on $\mathbb{S}^2 \subset \mathbb{R}^3$ \cite{schonefeld2005spherical}. Here, we consider all spherical harmonics up to degree $\ell = 3$. Namely,
\[
\{Y_\ell^m : 0 \le \ell \le 3,\ -\ell \le m \le \ell\},
\]
which yields $16$ basis functions. Our ground truth function is
\[
f(x) = \sum_{\ell=0}^{3} \sum_{m=-\ell}^{\ell} Y_\ell^m(x).
\]
The samples for this example are drawn from $\Omega$, defined as the lower quarter of the sphere, with additive Gaussian noise corresponding to $\mathrm{SNR}=30$~dB, where
\begin{equation}
\mathrm{SNR}_{\mathrm{dB}} \equiv 10\log_{10}\!\left(\frac{P_s}{P_n}\right),
\qquad
P_s \equiv \frac{1}{N}\sum_{i=1}^{N} f(x_i)^2,
\qquad
P_n \equiv \frac{1}{N}\sum_{i=1}^{N} \varepsilon_i^2.
\label{eq:snr_def}
\end{equation} 

The sampling and extrapolation regions are depicted in Figure~\ref{fig:sphere_setup}.

\begin{figure}[ht]
    \centering
    \begin{subfigure}[b]{0.48\linewidth}
        \centering
        \includegraphics[width=\linewidth]{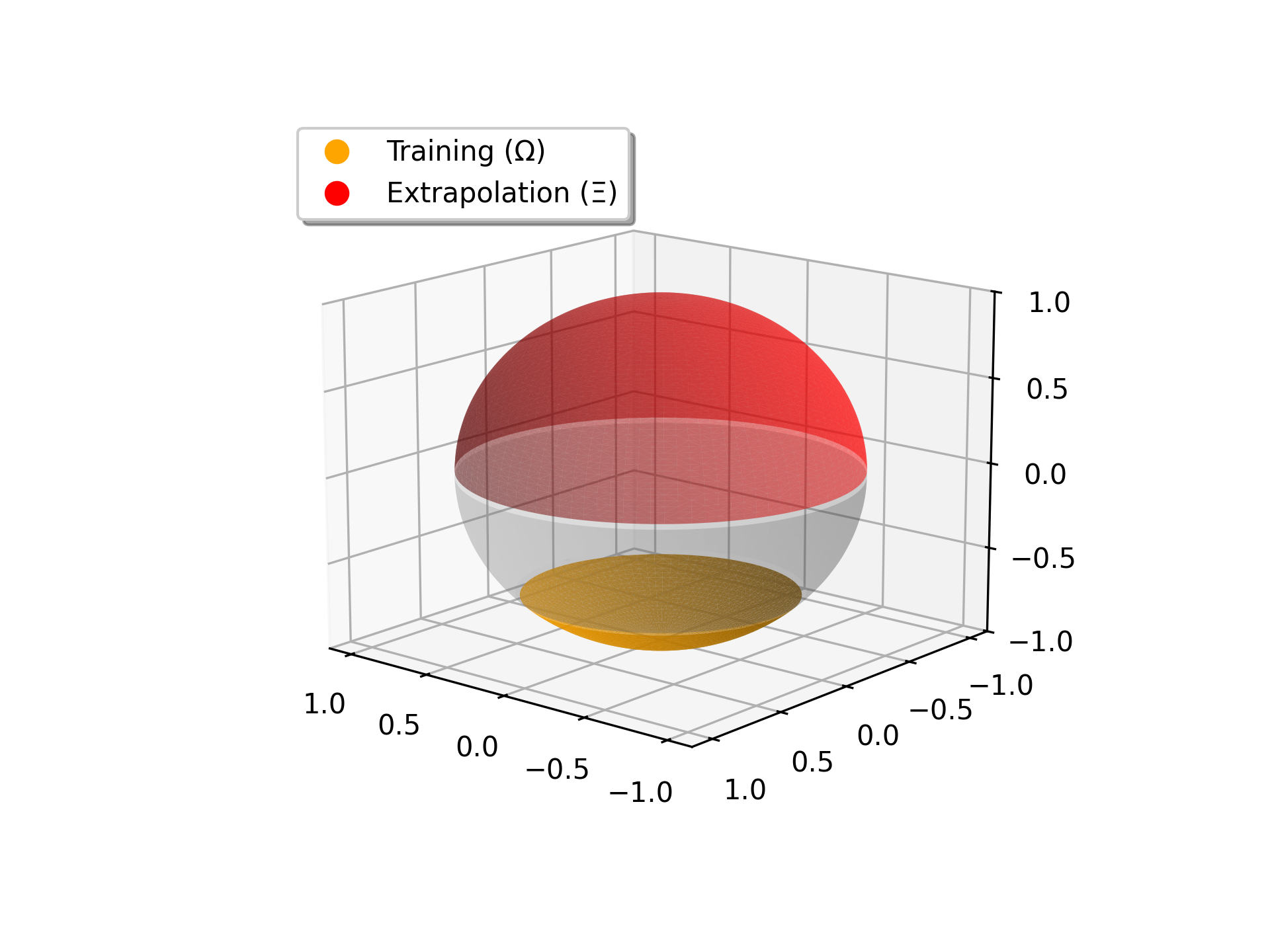}
        \caption{Sampling region $\Omega$ and extrapolation region $\Xi$ on the sphere.}
        \label{fig:sphere_setup}
    \end{subfigure}\hfill
    \begin{subfigure}[b]{0.48\linewidth}
        \centering
        \includegraphics[width=\linewidth]{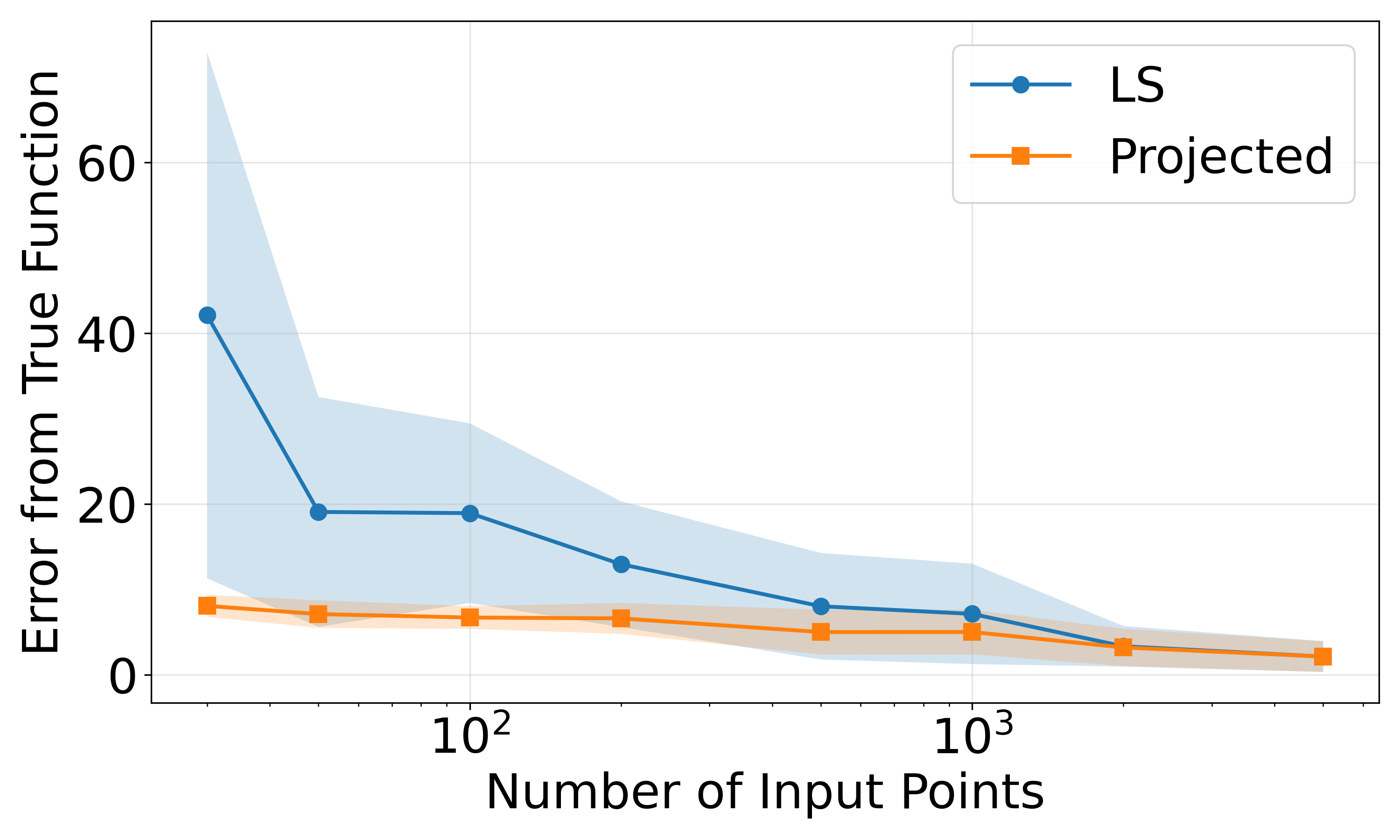}
        \caption{Extrapolation error on $\Xi$ for LS and its projection as a function of $|\Omega|$. Average and std across 10 runs.}
        \label{fig:sphere_ls_points}
    \end{subfigure}
    \caption{Spherical harmonics experiment: Sampling and extrapolation geometry (left) and extrapolation error versus $|\Omega|$ (right). A simple anchor constraint substantially improves accuracy, yielding performance comparable to that obtained with an order-of-magnitude increase in input samples.}
    \label{fig:sphere_experiment}
\end{figure}

Similar to the approach of Section~\ref{sec:easily_approximated_anchors}, we construct an anchor from the simple bounds
\[
\min f = -2.044948,\qquad \max f = 3.935627,
\]
by taking the midpoint
\[
a(x)\equiv a_0 := \tfrac{1}{2}(\min f + \max f) = 0.945339,
\]
and defining the feasible radius
\[
\delta := \sqrt{|\Xi|\,(\max f - a_0)} = 7.495540, \quad \|f - a\|_\Xi \le \delta .
\]
This bound is intentionally conservative and can be estimated numerically from a sufficiently dense grid over $\Xi$, without requiring access to the analytic harmonic expansion. It therefore serves as a simple range-based anchor rather than a sharp spectral certificate.

\begin{table}[ht]
    \centering
    \caption{Extrapolation error on $\Xi$ for LS and projected solutions as a
    function of the number of input points $|\Omega|$, mean and standard deviation across 10 runs.  The last two columns
    give the theoretical interval for the error reduction predicted by
    Theorem~\ref{theorem:5}.}
    \label{tab:sphere_ls_points}
    \begin{tabular}{rccccc}
    \toprule
    $|\Omega|$ &
    $\mathcal{E}_\Xi(g_{\mathrm{LS}})$ &
    $\mathcal{E}_\Xi(\widehat g)$ &
    $\mathcal{E}_\Xi(g_{\mathrm{LS}})-\mathcal{E}_\Xi(\widehat g)$ &
    lower bnd &
    upper bnd \\
    \midrule
     30   & 42.13$\pm$ 30.82 &  8.09$\pm$ 1.29 & 34.04 & 31.80 & 34.54 \\
     50   & 19.09$\pm$ 13.45 &  7.13$\pm$ 1.61 & 11.96 & 10.32 & 12.33 \\
    100   & 18.95$\pm$ 10.52 &  6.73$\pm$ 1.33 & 12.23 & 10.45 & 12.68 \\
    200   & 12.97$\pm$  7.37 &  6.63$\pm$ 1.81 &  6.35 &  5.12 &  6.63 \\
    500   &  8.05$\pm$  6.24 &  5.02$\pm$ 2.63 &  3.02 &  2.32 &  3.18 \\
   1000   &  7.16$\pm$  5.88 &  5.03$\pm$ 2.61 &  2.12 &  1.61 &  2.23 \\
   2000   &  3.36$\pm$  2.37 &  3.22$\pm$ 2.18 &  0.14 &  0.05 &  0.15 \\
   5000   &  2.15$\pm$  1.80 &  2.15$\pm$ 1.80 &  0.00 &  0.00 &  0.00 \\
    \bottomrule
\end{tabular}
\end{table}

To quantify the effect of the projection step, we interpret its benefit as the information gain relative to increasing the number of input samples. We therefore vary $|\Omega|$ between $30$ and $5000$, fit a least-squares (LS) model on $\Omega$, and project it onto the anchor-induced feasible set. Figure~\ref{fig:sphere_experiment} displays the resulting extrapolation error on $\Xi$ as a function of $|\Omega|$, with the corresponding numerical values reported in Table~\ref{tab:sphere_ls_points}.

For small data sets, the LS extrapolation error on $\Xi$ is highly unstable and can grow substantially, whereas the projected solution remains uniformly controlled by the anchor-induced envelope. For instance, with only $|\Omega|=50$ input points, the LS model attains on average $\mathcal{E}_\Xi(g_{\mathrm{LS}})=19.09\pm13.45$, while its projection yields $\mathcal{E}_\Xi(\widehat g)=7.13\pm1.61$, a reduction by more than a factor of $2$ using the same data and a simple range-based bound. Similar behavior occurs for $|\Omega|\in\{30,100,200,500,1000\}$: the LS error fluctuates between approximately $12$ and $70$, whereas the projected error remains on the scale of the anchor radius $\delta \approx 7.5$, and in most cases is even smaller.

This stabilization effect is comparable to increasing the number of LS input samples by roughly an order of magnitude. In all cases with nonzero improvement, the observed reduction lies within the theoretical interval predicted by Theorem~\ref{theorem:5}. Once the LS solution enters the feasible ball (here from some runs in $|\Omega|\gtrsim 500$, and most runs from $|\Omega|\gtrsim 2000$), the projection becomes inactive, and both curves coincide, precisely as guaranteed by the theorem.

\subsubsection{Probabilistic anchor function on a 2D Poisson problem}

We next evaluate the probabilistic certification framework in a structured PDE setting. Unlike the previous examples, where deterministic bounds were sufficient, we show that worst-case spectral certification can be overly conservative, and that the quantile-based probabilistic radius in Proposition~\ref{prop:prob_spec_bound} plays a decisive role in enabling meaningful projection corrections.

The domain  in this example is $D=[0,1]^2$, with a boundary-touching extrapolation patch $\Xi=[0.8,1]\times[0.7,1]$ and training region $\Omega=D\setminus \Xi$. We discretize $D$ on a uniform $N\times N$ grid with $N=200$, sample $M=200$ training points from $\Omega$, and corrupt the observations with i.i.d. Gaussian noise calibrated to $\mathrm{SNR}=35$ dB.

The target field $f$ is defined as the solution of the Dirichlet Poisson problem
\begin{equation}
-\Delta f(x,y)=s(x,y),\qquad f|_{\partial D}=0.
\end{equation}
This yields a smooth field with structured harmonic content. For fitting, we use the basis $\phi_{k_x,k_y}(x,y)=\sin(\pi k_x x)\sin(\pi k_y y)$ with $K=10$. We sample $T=25$ modes with $(k_x,k_y)\in\{2,\dots,K\}^2$ and random signs/magnitudes $|c^{(f)}_{k_x,k_y}|\in[0.1,7.0]$, and define the right-hand side in the same basis via
\begin{equation}
c^{(s)}_{k_x,k_y}=\pi^2(k_x^2+k_y^2)\,c^{(f)}_{k_x,k_y},
\end{equation}
so that $-\Delta f=s$ holds in the truncated basis.

\begin{figure}[t]
  \centering
  \begin{subfigure}{0.32\textwidth}
    \centering
    \includegraphics[width=\linewidth]{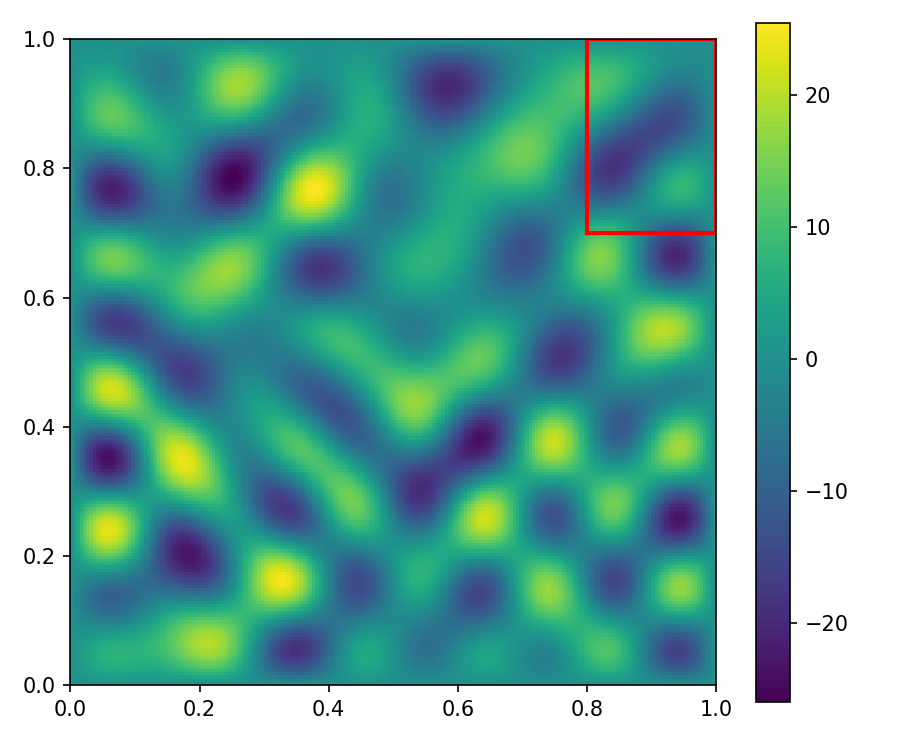}
  \end{subfigure}
  \begin{subfigure}{0.32\textwidth}
    \centering
    \includegraphics[width=\linewidth]{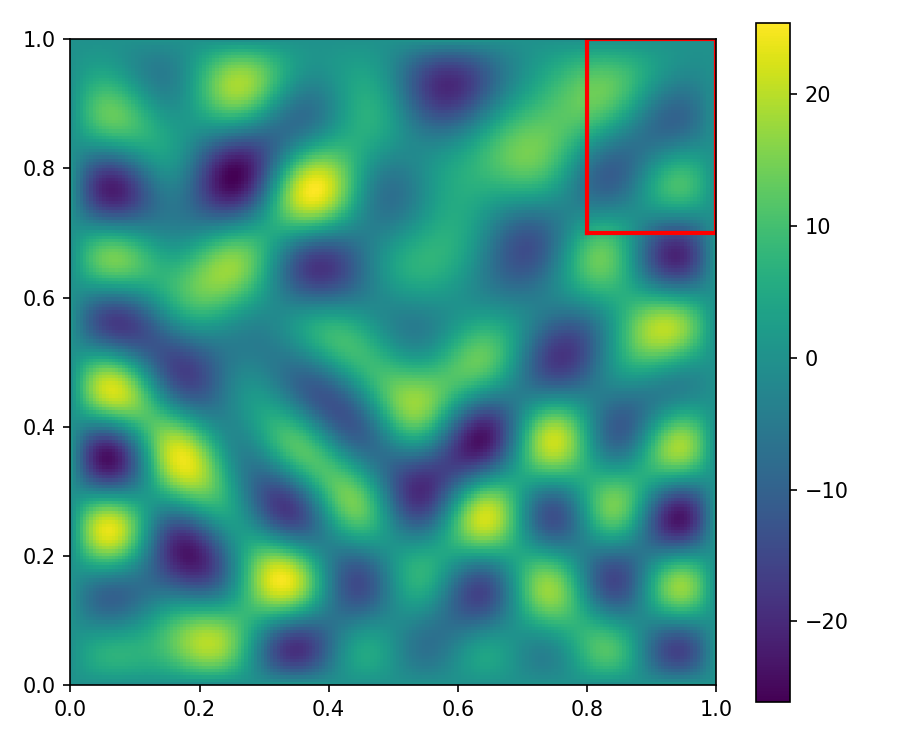}
  \end{subfigure}
  \begin{subfigure}{0.32\textwidth}
    \centering
    \includegraphics[width=\linewidth]{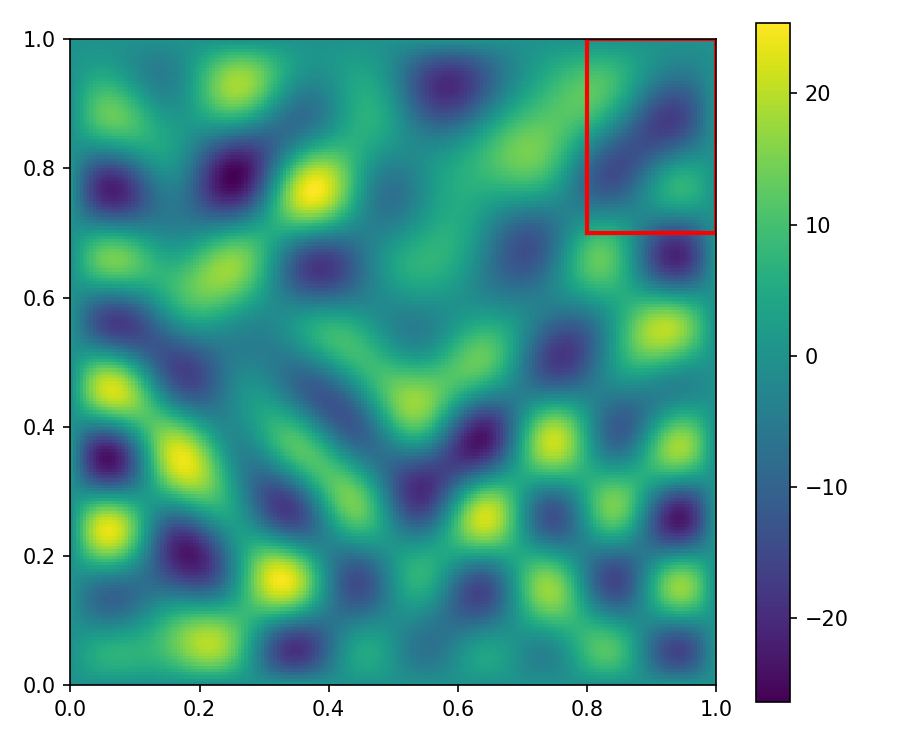}
  \end{subfigure}
  \caption{Value fields: ground truth \(f\) (left), baseline predictor \(g\) (center), and projection \(h\) (right). The extrapolation region \(\Xi\) is indicated in red.}
  \label{fig:exp_prob_anchor_pde2d_values}
\end{figure}

\begin{figure}[t]
  \centering
  \begin{subfigure}{0.49\textwidth}
    \centering
    \includegraphics[width=\linewidth]{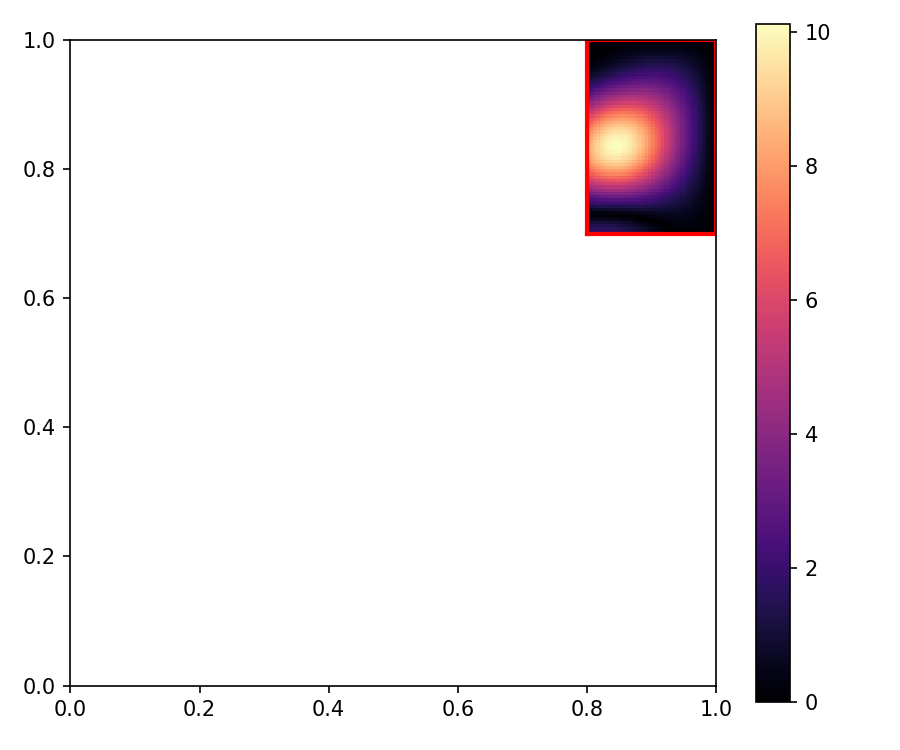}
  \end{subfigure}
  \begin{subfigure}{0.49\textwidth}
    \centering
    \includegraphics[width=\linewidth]{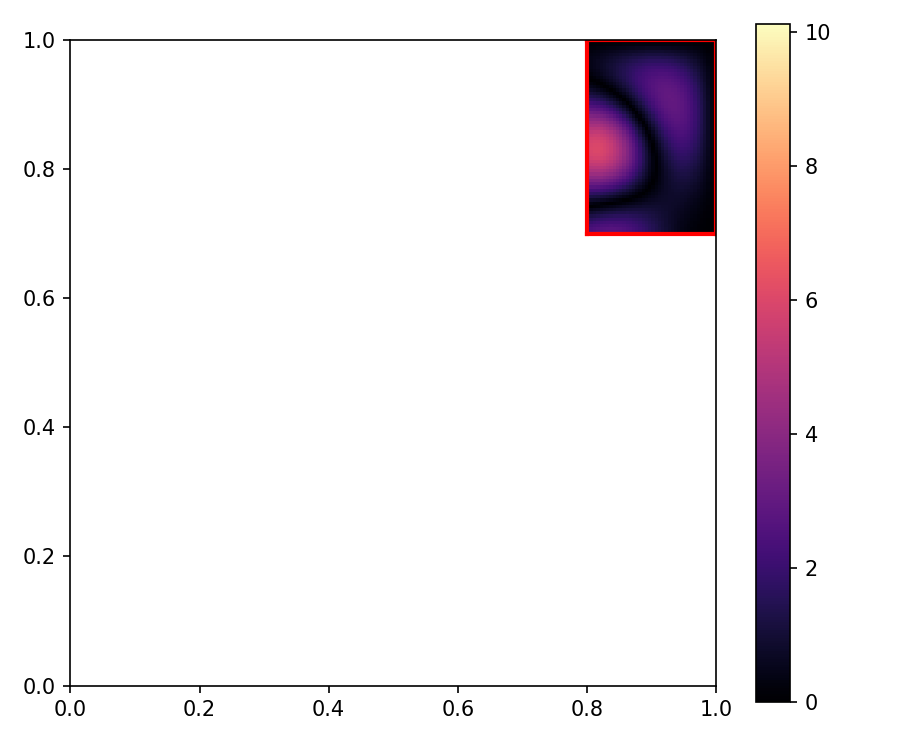}
  \end{subfigure}
  \caption{Extrapolation error on \(\Xi\): \(|f-g|\) (left) and \(|f-h|\) (right). The color scales are shared to make improvement visible.}
  \label{fig:exp_prob_anchor_pde2d_errors}
\end{figure}

We compute an anchor $a$ by LS and a baseline predictor $g$ by LASSO ($\alpha=2\times 10^{-4}$). We then form a probabilistic feasible set using the quantile-based construction of Definition~\ref{def:prob_anchor_function_sphere}, selecting a $70\%$ as a moderate-confidence setting that empirically tightens the feasible set while maintaining feasibility (see also Corollary~\ref{cor:almost_rank_one_beta_sphere}). As observed in Section~\ref{subsec:empirical_prob_kappa_sphere}, moving from worst-case to high-confidence quantiles yields significant reductions in the effective extrapolation constant, whereas further decreases in confidence provide diminishing returns relative to the loss in guarantee.

For this configuration, we obtain
\begin{equation*}
\lambda_{\max}(\tilde G_{\Xi})=6.1905\times 10^2,\quad \lambda_2(\tilde G_{\Xi})=2.4941\times 10^2,\quad
\kappa^{0.7}_{spec}=73.332,
\end{equation*}
and $\delta=3.4803$. Then, the extrapolation root errors on $\Xi$ are
\begin{equation*}
\mathcal{E}^{\mathrm{LASSO}}_{\Xi}=1.0651,\qquad \mathcal{E}^{\mathrm{Projection}}_{\Xi}=0.5148,
\end{equation*}
so the observed improvement is $\Delta \mathcal{E}_\Xi = 0.5503$. Theorem~\ref{thm:gurantee1} predicts an improvement interval,
\begin{equation*}
0.4071 \le \Delta \mathcal{E}_\Xi=0.5503 \le 0.8987,
\end{equation*}
which contains the observed gain.

Figure~\ref{fig:exp_prob_anchor_pde2d_values} visualizes the value fields over $D$: (a) the ground truth $f$, (b) the baseline predictor $g$, and (c) the projected solution $h$; the extrapolation region $\Xi$ is marked in red. Figure~\ref{fig:exp_prob_anchor_pde2d_errors} focuses on $\Xi$ and compares the absolute extrapolation errors $|f-g|$ (left) versus $|f-h|$ (right).

For comparison, using the worst-case spectral constant $\kappa_{\mathrm{spec}}$ yields $\delta = 10.1118$. In this regime, the LASSO predictor $g$ already lies within the associated feasible set, so the projection leaves the solution unchanged on $\Xi$. This comparison highlights the practical advantage of controllable probabilistic certification: by reducing conservatism in the radius, the feasible set becomes sufficiently informative to activate meaningful correction while preserving theoretical guarantees.

\section*{Declarations}

\noindent\textbf{Acknowledgement}  
NS is partially supported by NSF-BSF award 2024791, BSF award 2024266, and DFG award 514588180.

\noindent\textbf{Data availability}  
The geomagnetic data used in this study are derived from the publicly available WMM–2025 Gauss coefficients released by NOAA/NCEI~\cite{noaa_wmm_coefficients}.

\noindent\textbf{Code availability}  
The code used in this paper is available at:  
\url{https://github.com/guyhay94/anchor_based_extrapolation/tree/main}

\bibliographystyle{plain}
\bibliography{refs}
\appendix

\section{Appendix}

\subsection{Proof of Lemma~\ref{lemma:circle_point_circle_with_bounds}}
\label{appendix:lemma_circle_point_circle_with_bounds_proof}
\begin{proof}

\textbf{Existence of extrema.} The function $d$ is continuous on the compact set $\overline{\mathbb D}_R$, hence by the Extreme Value Theorem it attains a global minimum and a global maximum.

\medskip
\textbf{Step 1: interior critical points (KKT~\cite{boltyanski1999geometric} first step).}
On the open disk $\mathbb D_R=\{x^2+y^2<R^2\}$ the partial derivatives exist and are
\begin{align*}
 \frac{\partial d}{\partial x}(x,y) &= \frac{x-p}{\sqrt{(p-x)^2+y^2}} - \frac{x-R}{\sqrt{(R-x)^2+y^2}},\\[4pt]
 \frac{\partial d}{\partial y}(x,y) &= y\Bigg(\frac{1}{\sqrt{(p-x)^2+y^2}} - \frac{1}{\sqrt{(R-x)^2+y^2}}\Bigg).
\end{align*}
A stationary point in the interior must satisfy both partials equal to zero. From $\partial d/\partial y=0$ we have two alternatives:
\begin{enumerate}
\item $y=0$, or
\item the bracketed term vanishes, i.e. $\dfrac{1}{\sqrt{(p-x)^2+y^2}}=\dfrac{1}{\sqrt{(R-x)^2+y^2}}$, which is equivalent to $(p-x)^2=(R-x)^2$ and hence to
\[
 x=\frac{p+R}{2}.
\]
\end{enumerate}
Because $p>R$, the value $x=(p+R)/2$ lies strictly to the right of the circle and therefore cannot occur in $\mathbb D_R$. Thus any interior stationary point must satisfy $y=0$.

For $y=0$ we compute directly,
\[
 d(x,0)=|p-x|-|R-x|=(p-x)-(R-x)=p-R\quad\text{for }-R<x<R.
\]
So both partial derivatives vanish along the open horizontal diameter, and $d$ is constant there. These points form a non-isolated \emph{plateau of stationary points}. Moving slightly off the axis (take small $t\ne0$) yields
\begin{eqnarray*}
 \frac{\partial d}{\partial y}(x,t) &=& t\Bigg(\frac{1}{\sqrt{(p-x)^2+t^2}}-\frac{1}{\sqrt{(R-x)^2+t^2}}\Bigg) \\
 &<& t\Bigg(\frac{1}{|p-x|}-\frac{1}{|R-x|}\Bigg) <0 , \quad\text{for }t>0,
\end{eqnarray*}
because $|p-x|>|R-x|$. Hence, any small off-axis perturbation decreases $d$, while motions along the axis leave $d$ unchanged. Consequently, every interior point with $y=0$ attains the global maximum value $p-R$ (a non-strict/global plateau of maximizers).

\medskip
\textbf{Step 2: boundary (KKT Lagrange multipliers).}
Having exhausted interior candidates, any remaining extrema must lie on the boundary circle $g(x,y):=x^2+y^2-R^2=0$. The KKT conditions reduce to the Lagrange multiplier equations $\nabla d=\lambda\nabla g$ (the constraint qualification holds since $\nabla g\neq0$ on the circle). Put
\[
 r_P:=\sqrt{(p-x)^2+y^2}=PB,\qquad r_Q:=\sqrt{(R-x)^2+y^2}=QB.
\]
The Lagrange system becomes
\begin{align}
 -\frac{p-x}{r_P}+\frac{R-x}{r_Q} &= 2\lambda x, \label{eq:lag_x}\\[4pt]
 y\Big(\frac{1}{r_P}-\frac{1}{r_Q}\Big) &= 2\lambda y. \label{eq:lag_y}
\end{align}
If $y=0$ on the boundary then $(x,y)=(\pm R,0)$ and one computes $d(\pm R,0)=p-R$, recovering the upper bound. Assume henceforth $y\ne0$. Dividing \eqref{eq:lag_y} by $y$ gives
\[
 2\lambda=\frac{1}{r_P}-\frac{1}{r_Q}.
\]
Substituting into \eqref{eq:lag_x} and simplifying yields the equivalent relation
\[
 -\frac{p}{r_P}+\frac{R}{r_Q}=0 \quad\Longleftrightarrow\quad \frac{r_P}{r_Q}=\frac{p}{R}. \tag{*}
\]
On the circle $x^2+y^2=R^2$ we have
\[
 r_Q^2=(R-x)^2+y^2=2R^2-2Rx=2R(R-x),\qquad r_P^2=(p-x)^2+y^2=p^2-2px+R^2.
\]
Using (*) and squaring both sides produces
\[
 p^2-2px+R^2=\frac{p^2}{R^2}\big(2R^2-2Rx\big)=2p^2-\frac{2p^2}{R}x.
\]
Carrying out the algebra, we get:
\[
 p^2-2px+R^2=2p^2-\frac{2p^2}{R}x
 \;\Longrightarrow\; x\Big(\frac{2p^2}{R}-2p\Big)=p^2-R^2=(p+R)(p-R)
\]
\[
 \;\Longrightarrow\; 2px(p-R)=R(p+R)(p-R) \;\Longrightarrow\; x=\frac{R(p+R)}{2p}.
\]
Since $p>R$ we have $0<x_\star<R$, and then
\[
 y_\star=\pm\sqrt{R^2-x_\star^2}.
\]
Evaluate the distances at $B_\star=(x_\star,y_\star)$. From (*) we have $r_P=(p/R)r_Q$ and using $r_Q^2=2R(R-x_\star)$ one finds
\[
 r_Q(B_\star)=R\sqrt{\frac{p-R}{p}},\qquad r_P(B_\star)=\sqrt{p(p-R)}.
\]
Therefore, the value of $d$ at these boundary candidates is
\[
 d(B_\star)=r_P(B_\star)-r_Q(B_\star)=r_Q(B_\star)\Big(\frac{p}{R}-1\Big)=\frac{(p-R)^{3/2}}{\sqrt{p}}.
\]
This is the global minimum on $\overline{\mathbb D}_R$.

\medskip
We note that
\[
d(B_\star)=\frac{(p-R)^{3/2}}{\sqrt{p}}=(p-R)\sqrt{\frac{p-R}{p}} < p-R
\]
\end{proof}

\section{Derivation: generic sphere concentration bound for a quadratic form}
\label{app:derivation_generic_sphere_quantile}

Let $Q_{1-\delta}(X):=\inf\{t\in\mathbb{R}:\mathbb{P}(X\le t)\ge 1-\delta\}$ denote the
$(1-\delta)$--upper quantile of a real-valued random variable $X$.

\subsection{Generic concentration on the sphere to a quantile bound}

Generic concentration on the sphere (e.g.\ L\'evy's lemma; see \cite[Thm.~5.1.3]{vershyninHDP})
implies that if $s\sim\mathrm{Unif}(S^{d-1})$ and $f:S^{d-1}\to\mathbb{R}$ is $L$--Lipschitz
(with respect to the Euclidean metric), then for all $t\ge 0$,
\begin{equation}
\label{eq:vershynin_sphere_tail_app}
\mathbb{P}\!\left(\,|f(s)-\mathbb{E}f(s)|\ge t\,\right)
\;\le\;
2\exp\!\left(-c\,d\,\frac{t^2}{L^2}\right),
\end{equation}
where $c>0$ is an absolute constant.
(In \cite[Thm.~5.1.3]{vershyninHDP} the statement is given for
$X\sim\mathrm{Unif}(\sqrt{d}\,S^{d-1})$; applying it to $X=\sqrt{d}\,s$
and rescaling the Lipschitz constant yields~\eqref{eq:vershynin_sphere_tail_app} for the unit sphere.)

Specialize in the quadratic form
\[
f(s):=s^\top G s,
\qquad G\succeq 0,\quad \lambda_{\max}=\|G\|_{\mathrm{op}}.
\]
For $s,t\in S^{d-1}$,
\begin{multline*}
|f(s)-f(t)|
=
|(s-t)^\top G s+t^\top G(s-t)|
\\ \le 
\|s-t\|_2\bigl(\|G\|_{\mathrm{op}}\|s\|_2+\|G\|_{\mathrm{op}}\|t\|_2\bigr)
=
2\lambda_{\max}\|s-t\|_2,
\end{multline*}
and therefore, $\|f\|_{\mathrm{Lip}}\le 2\lambda_{\max}$.
Combining this with~\eqref{eq:vershynin_sphere_tail_app} gives
\[
\mathbb{P}\!\left(\,f(s)\ge \mathbb{E}f(s)+t\,\right)
\;\le\;
2\exp\!\left(-c\,d\,\frac{t^2}{4\lambda_{\max}^2}\right).
\]
Setting the right-hand side equal to $\delta$ and solving for $t$ yields: for some absolute constant
$C>0$,
\begin{equation}
\label{eq:generic_concentration_quantile_app}
Q_{1-\delta}\!\bigl(f(s)\bigr)
\;\le\;
\mathbb{E}[f(s)]
+
C\,\lambda_{\max}\sqrt{\frac{\log(2/\delta)}{d}}.
\end{equation}

Since $\mathbb{E}[f(s)]=\mathrm{tr}(G)/d$ (Remark~\ref{rem:mean_eigenvalue_consistency}),
\eqref{eq:generic_concentration_quantile_app} becomes
\begin{equation}
\label{eq:generic_concentration_quantile_trace_app}
Q_{1-\delta}\!\bigl(s^\top G s\bigr)
\;\le\;
\frac{\mathrm{tr}(G)}{d}
+
C\,\lambda_{\max}\sqrt{\frac{\log(2/\delta)}{d}}.
\end{equation}
Equivalently, if $c=\mathcal{E}s$ (so $\|c\|_2=\mathcal{E}$), then with probability at least $1-\delta$,
\[
c^\top G c
=
\|c\|_2^2\,s^\top G s
\;\le\;
\|c\|_2^2\left(
\frac{\mathrm{tr}(G)}{d}
+
C\,\lambda_{\max}\sqrt{\frac{\log(2/\delta)}{d}}
\right).
\]

\subsection{Comparison to the single-eigenvalue (Beta) model}

In the single eigenvalue case $G=\lambda_{\max}uu^\top$, Lemma~\ref{lem:rank_one_beta_sphere}
yields the \emph{exact} distribution $s^\top G s=\lambda_{\max}Z_d$ with
$Z_d\sim\mathrm{Beta}(\tfrac12,\tfrac{d-1}{2})$, and hence
\[
Q_{1-\delta}\!\bigl(s^\top G s\bigr)=\lambda_{\max}\,Q_{1-\delta}(Z_d).
\]
For any fixed confidence level $1-\delta$ (e.g.\ $1-\delta=0.95$), one has
$Q_{1-\delta}(Z_d)=\Theta(1/d)$, and therefore
\[
Q_{1-\delta}\!\bigl(s^\top G s\bigr)=\Theta\!\left(\frac{\lambda_{\max}}{d}\right).
\]
By comparison, the generic concentration bound~\eqref{eq:generic_concentration_quantile_trace_app}
specialized to rank one gives
\[
Q_{1-\delta}\!\bigl(s^\top G s\bigr)
\;\le\;
\frac{\lambda_{\max}}{d}
+
C\,\lambda_{\max}\sqrt{\frac{\log(2/\delta)}{d}}
=
\lambda_{\max}\Bigl(\Theta(\tfrac{1}{d})+\Theta(\tfrac{1}{\sqrt d})\Bigr),
\]
whose dominant term for large $d$ (fixed $\delta$) is $\Theta(\lambda_{\max}/\sqrt d)$.
Thus, already for fixed $\delta$ (including $\delta=0.05$), the concentration-based quantile
upper bound is asymptotically looser than the Beta-based quantile by a factor on the order of
$\sqrt d$.


\end{document}